\documentclass{amsart}

\usepackage{amsmath,amssymb,amsthm,a4wide}
\usepackage{url}
\usepackage{bbm}
\usepackage{slashed}
\usepackage{graphicx,tikz}
\usepackage{graphicx, subcaption, float}
\usepackage{mathrsfs}
\usepackage{hyperref}
\usepackage{enumitem}
\usetikzlibrary{shapes,arrows}
\usetikzlibrary{calc, positioning, shapes, through, intersections}
\newtheorem{theorem}{Theorem}
\newtheorem*{theorem*}{Theorem}
\newtheorem*{lemma*}{Lemma}
\newtheorem{proposition}[theorem]{Proposition}
\newtheorem{corollary}[theorem]{Corollary}
\newtheorem{lemma}[theorem]{Lemma}
\newtheorem{question}[theorem]{Question}
\newtheorem*{remark*}{Remark}

\theoremstyle{definition}
\newtheorem{definition}{Definition}
\newtheorem*{definition*}{Definition}

\theoremstyle{remark}

\newcommand{\rev}{}%{\color{red} }

\newcommand{\R}{\mathbb{R}}

\newcommand{\dd}{\mathop{}\!\mathrm{d}}

\newcommand\numberthis{\addtocounter{equation}{1}\tag{\theequation}}
\begin{document}

\title[]{Minimizing optimal transport for functions with fixed-size nodal sets}

\author{Qiang Du}
\address{Department of Applied Physics and Applied Mathematics, and Data Science Institute, Columbia University, 500 W120 Street, New York, NY 10027, USA}
\email{\href{qd2125@columbia.edu}{qd2125@columbia.edu}}
\author{Amir Sagiv}
\address{Department of Applied Physics and Applied Mathematics, Columbia University, 500 W120 Street, New York, NY 10027, USA}
\email{\href{as6011@columbia.edu}{as6011@columbia.edu}}
\maketitle

\begin{abstract}
Consider the class of zero-mean functions with fixed $L^{\infty}$ and $L^1$ norms and exactly $N\in \mathbb{N}$ nodal points. Which functions $f$ minimize $W_p(f_+,f_-)$, the Wasserstein distance between the measures whose densities are the positive and negative parts? 
We provide a complete solution to this minimization problem on the line and the circle, which provides sharp constants for previously proven ``uncertainty principle''-type inequalities, i.e., lower bounds on $N\cdot W_p (f_+, f_-)$.
We further show that, while such inequalities hold in many metric measure spaces, they are no longer sharp when the non-branching assumption is violated; indeed, for metric star-graphs, 
the optimal lower bound on $W_p(f_+,f_-)$ is not inversely proportional to the size of the nodal set, $N$.
Based on similar reductions, we make connections between the analogous problem of minimizing $W_p(f_+,f_-)$ for $f$ defined on $\Omega\subset\R^d$ with an equivalent optimal domain partition problem.
\end{abstract}

\section{Introduction}

The study of nodal sets (or zero sets) of functions is a classic topic in analysis. Loosely speaking, the key question is, for a given a function, {\em ``how big'' its nodal set $Z(f) = \{x ~  |  ~ f(x)=0\}$ is.}\footnote{Here we use the term ``nodal set'' for a broad class of functions (see Section \ref{sec:settings}), whereas in  other places it refers to the zero sets of a special class of functions, e.g., Laplacian eigenfunctions.} 

In recent years, this question has been connected
to the notion of {\em optimal transport}. The intuition is as follows: 
given a sufficiently regular metric Borel probability space $\Omega$ (e.g., a domain in~$\R ^d$ or a Riemannian manifold)
and a sufficiently regular {\rev and bounded}  real-valued function $f:\Omega \to \R $ with mean zero over $\Omega$, we consider its positive and negative parts $f_{\pm} (x)  \equiv \pm \max (\pm f(x), 0)$ as densities, and the nodal set $Z(f)$ as the interfaces of their supports.  If the nodal set is {\em small}, then  there should be some regions in the support of $f_+$ that cannot be near the support of $f_-$. Therefore, the cost of transporting $f_+$ to $f_-$ cannot be arbitrarily small.

A common metric which quantifies the notion of transport cost is the {\em Wasserstein-$p$ distance} $W_p (f_+, f_-)$. To make the connection between the nodal set and the transport cost more transparent, we ask:
\begin{equation}\label{eq:loose_q}
 \text{Given a nodal set of a certain size, which functions minimize $W_p(f_+, f_-)$?}
\end{equation}
The answer to this minimization problem should depend on the norms of $f$; first, the larger $\|f\|_1$ is, the more mass there is to transport. Second, the more localized the function is, as measured in this work by $\|f\|_{\infty}$, the more mass can be concentrated near the interface, thus decreasing the transport distance and cost. 

Before rigorously formalizing our question (see Question \ref{q:minim} in Section \ref{sec:settings}), we note that the motivation behind it originates from
the following type of inequalities: for a domain ~$\Omega$,
\begin{equation}\label{eq:uncert}
 W_1(f_+, f_-) \cdot {\rm Surf}\left\{x \in \Omega: f(x) = 0 \right\} \geq c(\Omega)  \left( \frac{\|f\|_{L^1}}{\|f\|_{L^{\infty}}} \right)^{\rho} \|f\|_{L^1} \, , \qquad \forall f\in C^0(\bar{\Omega}),
 \end{equation}
where $\rho >0$, $c(\Omega)>0$ is a constant that depends only on the domain $\Omega$, and ${\rm Surf}$ is a surface measure, e.g., the co-dimension $1$ Hausdorff measure. The first inequality of this type was proven by Steinerberger \cite{steinerberger2020metric} for two-dimensional domains such as the unit square $[0,1]^2$, with $\rho =1$. This result was then generalized to certain compact domains in $\R^d$ with arbitrary $d\geq 1$ by Steinerberger and the second author with $\rho = 4-1/d$, see \cite{sagiv2020transport}.
%{Amir: two comments, one is that we should avoid having footnote on page 1, so it is better if we can move this footnote to regular text. Second, the notion of "weaker" is a bit vague, at least, it might have ignored the change in the choices $c(\Omega)$ for different $\rho$. A smaller $\rho$ can give the result for a larger $\rho$ but with a possibly different constant factor. With this point in mind, I wonder if a better choice is to use  $c(\Omega,\rho)$. And if we want to avoid being vague, we can just directly say  
%"Note that when $\Omega$ has a finite
%volume ${\rm Vol}(\Omega)$, since $\|f\|_{L^1} \leq \|f\|_{L^{\infty}} \cdot {\rm Vol}(\Omega)$, the inequality \eqref{eq:uncert} with a smaller $\rho$ implies the result for a larger $\rho$ with a possibly different constant factor $c(\Omega,\rho)$."}
Subsequently, Carroll, Massaneda, and Ortega-Cerd\`{a}~\cite{carroll2020enhanced} improved the exponent to $\rho = 2-1/d$, and Cavalletti and Farinelli proved that the optimal $\rho=1$  is true in all dimensions~\cite{cavalletti2021indeterminacy}. The analysis in \cite{cavalletti2021indeterminacy} takes place at a much larger class of metric measure spaces that satisfy the Curvature-Dimension (${\rm CD}(K,N)$) condition, where the nodal set can be measured using the notion of Perimeter \cite{ambrosio2014equivalent, miranda2003functions}. That result was then also proven for ${\rm RCD}(K,\infty)$ metric spaces \cite{ambrosio2015riemannian, ambrosio2014metric} by De-Ponti and Farinelli for all $p\geq 1$ \cite{deponti2022indeterminacy}.

Despite the great level of generality of the above results, some fundamental questions still remain: is the inequality \eqref{eq:uncert} with an exponent $\rho=1$ sharp? First, we do not know what the optimal constant~$c(\Omega)$~is. More fundamentally, one may ask whether a {\em multiplicative} inequality is the natural one. Indeed, some works studied related additive inequalities \cite{buttazzo2020wasserstein, candau2022existence, novack2021least, xia2021existence} regarding transport of the Lebesgue measure between disjoint domains. \footnote{\rev In particular, \cite[Proposition 2.5]{candau2022existence} and \cite[Lemma 2.8]{novack2021least} consider the overall transport outside of a bounded set $E\subset \R ^d$ and prove that a lower bound on the transport cost must be inversely proportional to the perimeter of $E$, i.e., if $F\subseteq E^c$ with volume $|E|=|F|=1$, then $W_p(\mathbbm{1}_E,\mathbbm{1}_F)>C(d)/{\rm Per}(E)$ for some universal $C(d)>0$.}  

But even the most general results regarding inequalities of the form \eqref{eq:uncert} require an essentially non-branching property \cite{cavalletti2021indeterminacy, deponti2022indeterminacy}: all ${\rm RCD}(K, \infty)$ spaces satisfy this property \cite{rajala2014non}, whereas the result of \cite{cavalletti2021indeterminacy} is proven only to ${\rm CD}(K,N)$ spaces which satisfy this property. {\em A primary takeaway of our study is that it is indeed necessary in order for inequalities like \eqref{eq:uncert} to be sharp.}

\subsection{Main results}\label{sec:mainresults}
In this paper, we formalize the minimization problem \eqref{eq:loose_q} and establish a strategy to its solution. We find the minimizers on three families of one-dimensional domains: a line (or an interval), a circle, and metric star graphs. The analysis of these domains leads us to the following key conclusions:
\begin{figure}[h]
\centering
{\includegraphics[scale=1.5]{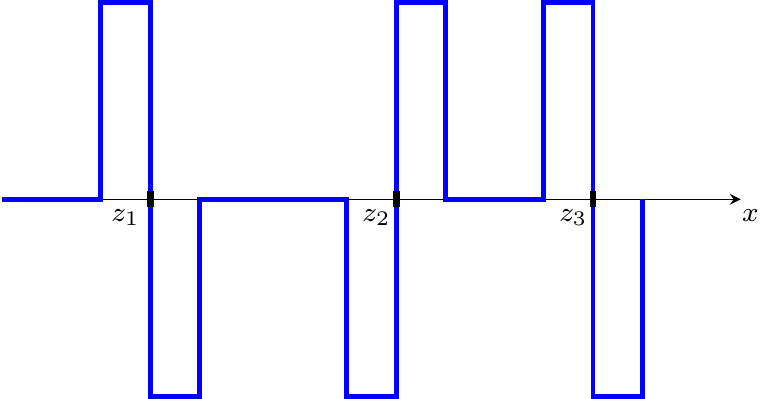}}
\caption{
An exemplary minimizer of $W_p(f_+,f_-)$ with $3$ nodal points on an interval.}
\label{fig:steps_opt}
\end{figure}
\begin{itemize}
    \item {\bf Generalized nodal sets for $L^1 \cap L^{\infty}$ functions.} Common to all these     settings are that the minimizers {\em are step functions}, and therefore do not have nodal sets in the usual sense. Hence, it is crucial to extend the class of admissible functions and notion of nodal sets of continuous functions to a broader class of $L^1 \cap L^{\infty}$ functions. Furthermore, it may be expected that this would be the right formulation through which to explore the minimizers in higher dimensions (see Section \ref{sec:outlook}).
    \item \textbf{Strategy.} Our work establishes a strategy to approach the minimization problem under consideration through a series of reductions of the class of feasible functions. The reduction strategy consists of three stages: {\bf (1)} We reduce the functional minimization problem to a geometric one. Fixing the nodal domain, i.e., the interface between the positive and negative parts (see Definition \ref{def:nodal}), we show that it is always preferable to allocate the $L^1$ mass near the nodal domain. Thus, given a function $f$ we find a new function $g$ which is (up to a sign and scaling) the  indicator function of {\em a subset} of the supports of $f_{\pm}$, and for which the overall transport cost is cheaper. {\bf (2)} We find that there is always a preferable configuration of the mass and the nodal points such that the optimal transport plan only couples adjacent intervals. {\bf (3)} Steps (1) and (2) reduce the infinite-dimensional variational problem to a  constrained finite dimensional problem, which one can solve by elementary means.
 \item {\bf The minimizers} By posing the question as a minimization problem, our work goes beyond previous works on the subject since we gain insight into the minimizers themselves. In the simplest case, the interval (or for every smooth non-intersecting curve), the minimizers are step functions whose only possible values are either $\pm \|f\|_{\infty}$ or~$0$, as can be seen in Fig.\ \ref{fig:steps_opt}. Thus, these are functions for which the optimal transport between $f_+$ and $f_-$ is local, across a single interface. Next, even though the circle introduces a new {\em global} structure, we show that the nature of the problem (and the minimizers) do not change.

    \item {\bf Sharp multiplicative inequalities for the interval and the circle.} 
  The complete solution of the minimization problem
     allows us to explore the sharpness of the multiplicative  ``uncertainty principle''  \eqref{eq:uncert}. For an interval,
     we find that multiplicative inequalities are
     optimal (Section \ref{sec:interval}). The sharp inequality for $p\geq 1$ is
    $$  W_p(f_+, f_-)  \cdot |Z(f)|\geq 2^{-1-\frac1p} \left(\frac{\|f\|_1}{\|f\|_{\infty}} \right) \|f\|_1^{\frac1p} \, .$$
    We further show that the same conclusion is not limited to an interval domain but also remains valid for a circle (see Section \ref{sec:circle}). This, in turn, leads us to seek for an example of domains, e.g., a metric star graph, for which the sharp lower bound is not expected to {\rev be} multiplicative (see Section \ref{sec:trees}).

    \item {\bf Metric star graphs.} We carry our strategy to the more complicated case of metric star graphs (e.g., we need to solve Kontorvich' problem and not Monge's); Indeed, as before, Theorem \ref{thm:starequi} reduces the general minimization problem over $L^1\cap L^{\infty}$ into a finite-dimensional constrained optimization problem. The key difference, however, is that in contrast to the line and the circle,  the optimal relation between the minimal transport cost and the other factors under consideration is {\em not multiplicative} on star graphs. The introduction of a node yields optimal lower bounds which depend on the geometry of the domain. For example, in the case of a star with $D\in \mathbb{N}$ sufficiently long edges, 
    $$W_1(f_+, f_-)\tilde{N} \geq \frac14 \frac{\|f\|_1}{\|f\|_{\infty}}\|f\|_1 \, , \qquad \tilde{N} = |Z(f)|-1 + \left\{ \begin{array}{cc}
         \frac{D}{2} \, ,& D~{\rm even} \, , \\
         \frac{(D+1)(D-1)}{2D} \, , & D~{\rm odd} \, .
    \end{array} \right. $$
    When some of the edges are not sufficiently long, even more complicated forms of inequalities emerge. Our conclusions for star graphs echo and contrast the analysis of \cite{cavalletti2021indeterminacy}, where the multiplicative lower bound \eqref{eq:uncert} is proved for a broad class of non-branching spaces. Since stars, and metric graphs in general, are branching spaces (see also \cite{erbar2021gradient}), one might expect a different type of results, as we indeed prove in Section \ref{sec:trees}.
      
\end{itemize} 
The strategy and issues identified in this work are expected to be generalizable to high dimensional settings. The one-dimensional settings help establish the strategy in its simplest form, and allow us to reach easy-to-compute and precise constants. Furthermore, by considering graphs, we are able to expose the key challenges in going beyond all previous works on non-branching spaces. In Section \ref{sec:outlook}, we outline the conjectures and key challenges in generalizing our strategy to multi-dimensional domains, as well as to general metric graphs.

\subsection{Structure of the paper} Preliminaries, definitions, and the formulation of our minimization problem (Question \ref{q:minim}) are given in Section \ref{sec:settings}. Then, the problem is solved for the interval (Theorem~\ref{thm:w1_min}) and the circle (Theorem \ref{thm:wpcir}) in Sections \ref{sec:interval} and \ref{sec:circle}, respectively. For star graphs, in Section \ref{sec:trees} we re-formulate Question \ref{q:minim} as a finite-dimensional constrained minimization problem (Theorem \ref{thm:starequi}) and solve it for a number of special cases. Finally, an outlook on the problem in multiple dimensions and on general metric graphs is presented in Section \ref{sec:outlook}.
\section{Settings}\label{sec:settings}
\subsection{Nodal Sets.}
Given a one-dimensional domain $I$, constants $c_{\infty}, c_1>0$, and $N\in \mathbb{N}$, define
\begin{equation}\label{eq:setX}
X= X(c_{\infty}, c_1, N, I) \equiv \left\{ f:I\to \R~~{\rm measurable} ~~ {\rm such}~{\rm that} ~~ \begin{array}{l}
\|f\|_{\infty} = c_{\infty} \, , \\
\|f\|_{1} = c_1 \, , \\
\left| Z(f)  \right| = N \, , \\
\int_I f(x) \, dx = 0 \, .
\end{array}   ~ ~  \right\} \, ,
\end{equation}
where $Z(f)$ is the set of points where $f$ changes its signed, defined as follows:
\begin{definition}[effective nodal set]\label{def:nodal}

    Consider a measurable $f:I\to \R$ with finite $L^1$ and $L^{\infty}$ norms.
    \begin{itemize}
    \item Define
    $$Z_1(f) \equiv   \partial \{x~|f(x)>0\} \cap \partial \{x~|f(x)<0\}  \, .$$
    \item Let $\mathcal{Z}(f) \equiv \pi_0 \left(f^{-1}(\{0\})\right)$, the set of all connected components of $f^{-1}(\{0\})$.
    \item Let $\mathcal{Z}'(f)\subseteq \mathcal{Z}(f)$ be the set of all elements of $\mathcal{Z}(f)$ whose closure intersects both  $\partial \{x~|f(x)>0\}$ and  $\partial \{x~|f(x)<0\}$.
    \item Let $Z_2(f)$ be a set which consists of a unique point $x\in J$
    for every $J\in \mathcal{Z}' (f)$.

    \item Finally, define the effective nodal set as $Z(f) \equiv Z_1(f) \cup Z_2(f)$.
\end{itemize}
    %the set $Z(f)\subset I$ such that for each $x\in Z(f)$, there exists $\Delta>\delta\geq 0$, possibly depending on $x$, such that $f(y) =0$ on $[x-\delta,x+\delta]$, {\em and} either:
    %\begin{itemize}
     %   \item $f(y)< 0$ for almost every $y\in (x-\Delta ,x-\delta)$ and $f(y)> 0$ for almost every $y\in (x+\delta,x+\Delta)$; {\em or}
      %  \item  $f(y)> 0$ for almost every $y\in (x-\Delta ,x-\delta)$ and $f(y)< 0$ for almost every $y\in (x+\delta,x+\Delta)$.
    %\end{itemize}
\end{definition}

\begin{figure}[h]
\centering
{\includegraphics[scale=1.5]{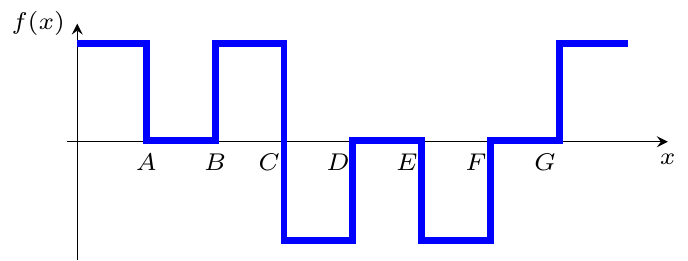}}
\caption{
Under Definition \ref{def:nodal}, the portrayed function has $|Z(f)|=2$. }
\label{fig:zero_nz}
\end{figure}

Our intention is to define $Z(f)$ in a way which exactly captures every sign change of $f$ once. Let us see that this is indeed achieved by the definition: consider, for example, the function in Fig.\ \ref{fig:zero_nz}. While $f=0$ on $[A,B]$, and so $[A,B]\in \mathcal{Z}(f)$, it is not an element of $\mathcal{Z}'(f)$, since it does not intersect with $\partial \{ f<0\}$. Hence, no point in $[A,B]$ is contained in $Z(f)$. The same is true for $[D,E]$. By definition, $C\in Z_1(f)$. Lastly, $[F,G]\in {\rev \mathcal{Z}'(f)}$, and therefore $F\in Z_2 (f)$. Overall, $f$ changes its sign exactly twice and indeed $|Z(f)|=|\{ C, {\rev F}\} |=2$. 

We contrast our definition with those used in previous works:
\begin{itemize}
    \item In \cite{sagiv2020transport, steinerberger2020metric}, the objects of study are continuous functions, and therefore the interface between the supports of $f_+$ and $f_-$ is always contained in the set $f^{-1}(\{0\})$. For general measurable functions, therefore, Definition~\ref{def:nodal} of the effective nodal set {\em need not} coincide with $ f^{-1}(\{0\})$ for $f\in C^0 (I)$ (consider e.g., $f(x)=x^2$ on $[-1,1]$). Moreover, for functions which are merely bounded but not continuous, these interfaces are not necessarily in $f^{-1}(\{0\})$, and hence the somewhat more complicated form of Definition \ref{def:nodal}.
    
   \item The authors of \cite{cavalletti2021indeterminacy} considered the quantity ${\rm Per}(\{x~|~f(x)>0\})$. First note that, in Euclidean settings, the Perimeter coincides with the Hausdorff measure $\mathcal{H}^{d-1}$ of the reduced boundary of $\{f>0\}$, and in particular for $d=1$ this is just the cardinality of the set. 
   
   Furthermore, the zero set in \cite{cavalletti2021indeterminacy}, $\partial\{x|f(x)>~0\}$, includes points where $f$ does not changes signs, e.g., $A$ and $B$ in Figure \ref{fig:zero_nz}. This is not an issue when seeking lower bounds of the type \eqref{eq:uncert}, as adding points which are not interfaces between $f_+$ and $f_-$ can only {\em increase} the left hand side of \eqref{eq:uncert}. Since we solve minimization problems on functions with exactly $N$ sign changes, it is easier if we avoid such issues, even at a cost of a slightly more elaborated definition.
\end{itemize}

\subsection{Optimal Transport and the Wasserstein distance.}
We briefly recall the definition and certain key properties of the {\em Wasserstein-$p$ distance}. We refer to \cite{santa2015optimal, villani2003topics} for a more comprehensive treatment of this topic.  Let $p\geq 1$, $\Omega \subseteq \R ^n$ a Borel set, and denote by $\mathbb{P}_p(\Omega)$ the set of all Borel probability measures on $\Omega$ with finite $p$-th moments. Define the Wasserstein-$p$ distance between two measures 
 $\mu_1, \mu_2 \in \mathbb{P}_p(\Omega)$ as
 \begin{equation}\label{eq:wasdef}
    W_p(\mu_1, \mu_2) \equiv \inf\limits_{\gamma \in \Gamma(\mu_1,\mu_2)} K_p^{\frac{1}{p}}(\gamma) \, , 
    \qquad K_p(\gamma) \equiv \int\limits_{\Omega \times \Omega} |x-y|^p \, \dd \gamma (x,y) \, ,
\end{equation}
where $|x-y|$ is the geodesic distance on $I$ and $\Gamma (\mu_1, \mu_2)$ is the set of all Borel probability measures on $\Omega \times \Omega$ with marginals $\mu_1$ and $\mu_2$, i.e.,
\begin{equation}\label{eq:Pi}
    \mu_1(A) = \gamma(\Omega\times A) \, , \qquad \mu_2 (A)=  \gamma(A\times \Omega) \, ,
\end{equation}
for any $\gamma \in \Gamma(\mu_1, \mu_2)$ and any Borel set $A\subseteq \Omega$. 

A measure $\gamma$ is often called a {\em transport plan} and $\gamma$ minimizing $K_p$ is said to solve the {\em Kantorovich problem}. In some cases, there exists an optimal transport {\em map}, a function $T:\Omega\to \Omega$ which solves the so-called {\em Monge problem}:
   \begin{equation}\label{eq:monge}
       \inf\limits_{\{T  ~~|~~ T_{\#}\mu_1 = \mu_2 \} } K_p^{\frac{1}{p}}(T) \, , 
    \qquad K_p(T) \equiv \int\limits_{\Omega} |x-T(x)|^p \, d\mu_1 (x) \, ,
    \end{equation}
    where by $T_{\#}\mu_1$ we mean the pushforward of $\mu_1$ by $T$, i.e., the measure which assigns to any Borel set $A\subseteq \Omega$ the measure $T_{\#}\mu_1 (A) = \mu_1 (T^{-1}(A))$. Any map that pushes $\mu_1 $ to $\mu _2$ induces a transport plan $({\rm id}, T)_{\#}\mu_1\in \Gamma(\mu_1, \mu_2)$, and the transport cost is unambiguously defined, i.e., we can write $K_p (T)$ as a shorthand for $K_p(({\rm id}, T)_{\#}\mu _1)$. 
    
    On an {\em interval} $I\subseteq \R$, for any $p\geq 1$ and any two atomless measures, an optimal transport plan is induced by an optimal transport {\em map} \cite[Theorem 2.9]{santa2015optimal}, i.e., $W_p^p(\mu_1, \mu_2) = K_p(T)$, with a monotonically increasing $T$ defined by
    \begin{equation}\label{eq:Tmono}
    T= F_{\mu_1}^{-1} \circ F_{\mu_2} \, ,
    \end{equation} where $F_{\mu}(y) \equiv \mu(-\infty, y)$ is the cumulative distribution function (CDF), and the inverse is taken in the generalized sense, $F^{-1}_{\mu}(x) \equiv \inf \{t\in \R ~|~ F_{\mu}(t)\geq x \}$.\footnote{The statement holds more generally for the optimal transport with respect to any convex cost function $h(x-y)$ on the line, but we will not pursue this level of generality here.} For $p>1$, this map is also unique~\cite{santa2015optimal}. Hence, the Wasserstein-$p$ distance has the much simpler form
\begin{equation}\label{eq:wp_icdf}
    W_p (\mu_1, \mu_2) = \|F_{\mu_1}^{-1} - F_{\mu_2}^{-1} \|_{L^p(\R)} \, .
\end{equation}
In the particular case of $p=1$, one gets the more straightforward formula with the CDFs (and not their inverses) $W_1(\mu_1, \mu_2) = \|F_{\mu_1}-F_{\mu_2}\|_{L^1(\R)}$, see \cite{salvemini1943sul, vallender1974calculation}.

Our main question can be rigorously formalized as
\begin{question}\label{q:minim}
For a one-dimensional domain $I$ with nonzero measure $|I|>0$, $p\geq 1$, $c_{\infty}>0$,
$c_1\in (0, c_{\infty}|\Omega|]$, and $N\in \mathbb{N}_+$, what are the minimizers and the minimum value of the
%the the
minimization problem
\begin{equation}\label{eq:minprob}
    \min_{ f\in X(c_{\infty}, c_1, N, I) }
    W_p (f_+, f_-) \,.
    %~~ |~~ f\in X(c_{\infty}, c_1, N, I) ~\} \, .
    \end{equation}
\end{question}
By $W_p (f_+ , f_-)$, we mean the $W_p$-distance between the measures whose densities are $f_+$ and $f_-$. Note that $c_1\in (0, c_{\infty}|\Omega|]$ and $N>0$ are specified in the statement of Question \ref{q:minim} as {\rev necessary and sufficient} condition for $X(c_{\infty}, c_1, N, I)\neq \emptyset$.  

A crucial ingredient to study the above minimization problem is to 
establish an equivalence formulation that provides a characterization of the minimizers to the original problem. Connected to this, 
we present a sub-class of functions, defined below:
\begin{definition}\label{def:Xs}
Let $I$ be a one dimensional domain (a curve or a metric graph). Denote by $X_s =X_s(c_{\infty},c_1,N,I)$ the set of step-functions $f\in X(c_{\infty},c_1,N,I)$  such that $f=~\pm c_{\infty}$ only on intervals adjacent to points in $Z(f)$ and $0$ everywhere else; see, e.g., {\rev Figure \ref{fig:steps_opt} and Figure~\ref{fig:opt_int}(B)}.
\end{definition}

\section{The Interval}\label{sec:interval}

In this section, we study the case of a nonempty interval $I = (0,L)$ with $L>0$.
Our strategy to answer Question \ref{q:minim}, here and throughout this paper, is that of {\em optimization;} for every candidate function $f\in X$, we attempt to construct $g\in X$ such that $W_p(f_+,f_-) > W_p(g_+,g_-)$. The minimizers will therefore be the only functions {\rev for which further optimization is not possible,} and we will show that, by construction, those are also global minimizers.

\begin{lemma}\label{lem:steps}
Let $I=(0,L)$.
For every $f\in X= X(c_{\infty},c_1,N,I)$, denote $Z(f) = \{z_1 , \ldots , z_N \}$ with $z_i <z_{i+1}$ for all $1\leq i <N$. Then, there exists a function $g \in X_s$ such that
\begin{enumerate}[label=(\roman*)]
\item $Z(f) = Z(g)$
\item For any $z_i \in Z(g)$, $g\geq 0$ on $I_i=(z_i, z_{i+1})$ if and only if $f\geq 0$ there, and
\[\int_{I_i} g(x) \,dx = \int_{I_i} f(x) \, dx.\]

\item $W_p (f_+,f_-) \geq W_p (g_+, g_-)$ for any $p\geq 1$, with strong inequality if $f\not\in X_s$.
\end{enumerate} 
\end{lemma}
\begin{remark*}
Even though the optimal transport plan in this case is given by the monotonic map~\eqref{eq:Tmono}, we will work in this proof with a general coupling $\gamma \in \Gamma (f_+, f_-)$. This level of generality shows that at least the geometric nature of the problem extends to the following more general setting: let $\Omega$ be a simple curve, $c(x,y)=h(|x-y|)$ where $h:\R_+ \to \R_+$ is a monotonically increasing function in the geodesic distance for two points $x,y\in \Omega$, and define the $c$-transport cost as 
\begin{equation}\label{eq:Kc}
K_c (T) = \int_{\Omega \times \Omega} c(x,T(x)) \, , d\mu (x) \, .
\end{equation}
Then, defining the optimal transport with respect to $c(x,y)$ analogously to \eqref{eq:wasdef}, Lemma \ref{lem:steps} still holds. Also, it will allow us to generalize this statement immediately to star graphs in Section \ref{sec:trees}.

\end{remark*}
\begin{proof}
%Denote $Z(f) = \{ z_1, \ldots, z_N\} \subset (0,L)$, and define $I_j = (z_j, z_{j+1})$ with $z_0= 0$ and $z_{N+1} = L$. 
Consider an optimal transport plan between $f_+$ and $f_-$, i.e., a measure $\gamma \in \Gamma (f_+, f_-)$ such that~$K_p(\gamma) = W_p^p(f_+, f_-)$, see equation \eqref{eq:Pi}. The intuition is that for each $0\leq j \leq N$, some of the mass on $I_j$ has to be transported to the left by $\gamma$, and some has to be transported to the right, see Fig.\ \ref{fig:pre_opt_int}.\footnote{The red dotted line, separating between the left and right transported parts in $I_j$, is vertical in Fig.\ \ref{fig:pre_opt_int}, since the optimal transport in this case is given by a map. In the case when the optimal transport is given by a coupling, this separation would be better depicted by a curve $h(x)$ with $0\leq h(x)\leq f_+ (x)$.} It is therefore less costly to have those respective masses already concentrated near the endpoints of~$I_j$, see Fig.\ \ref{fig:post_opt_int}. 
\begin{figure}[h]
 \begin{subfigure}[t]{.48\textwidth}
            \caption{}
            \includegraphics[width=.95\linewidth]{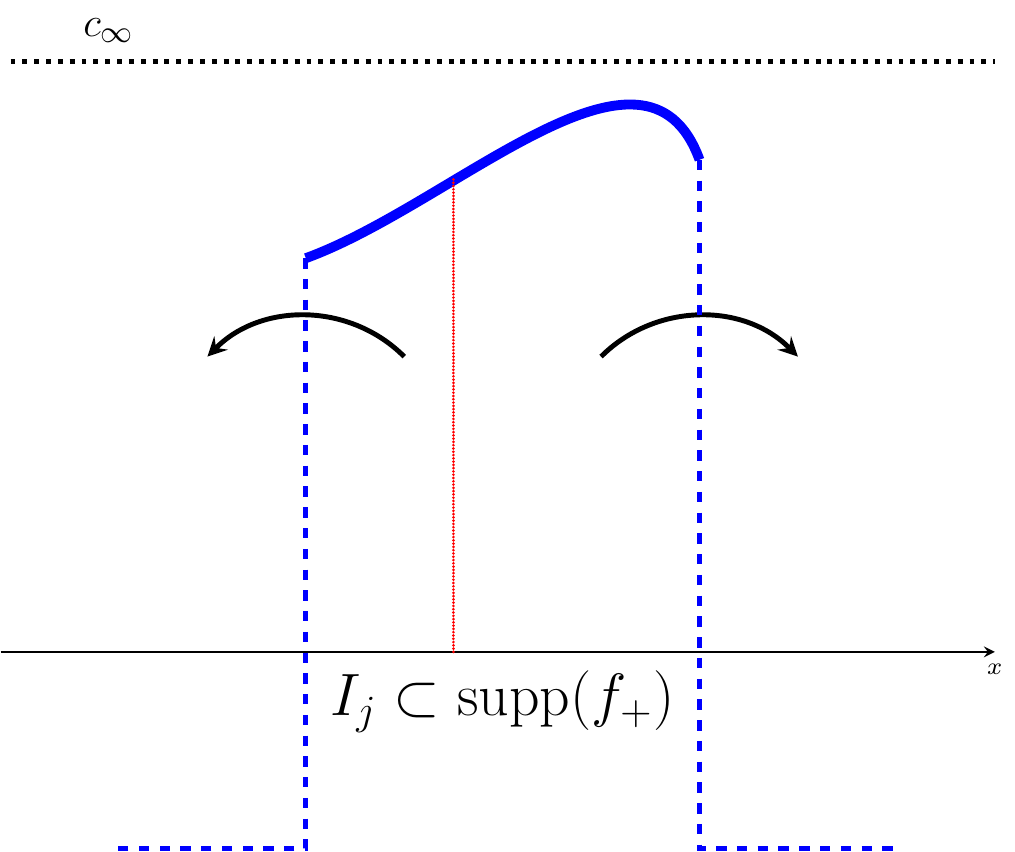}
            \label{fig:pre_opt_int}
        \end{subfigure}
        \begin{subfigure}[t]{ .48\textwidth}
            \caption{}
            \includegraphics[width=.95\linewidth]{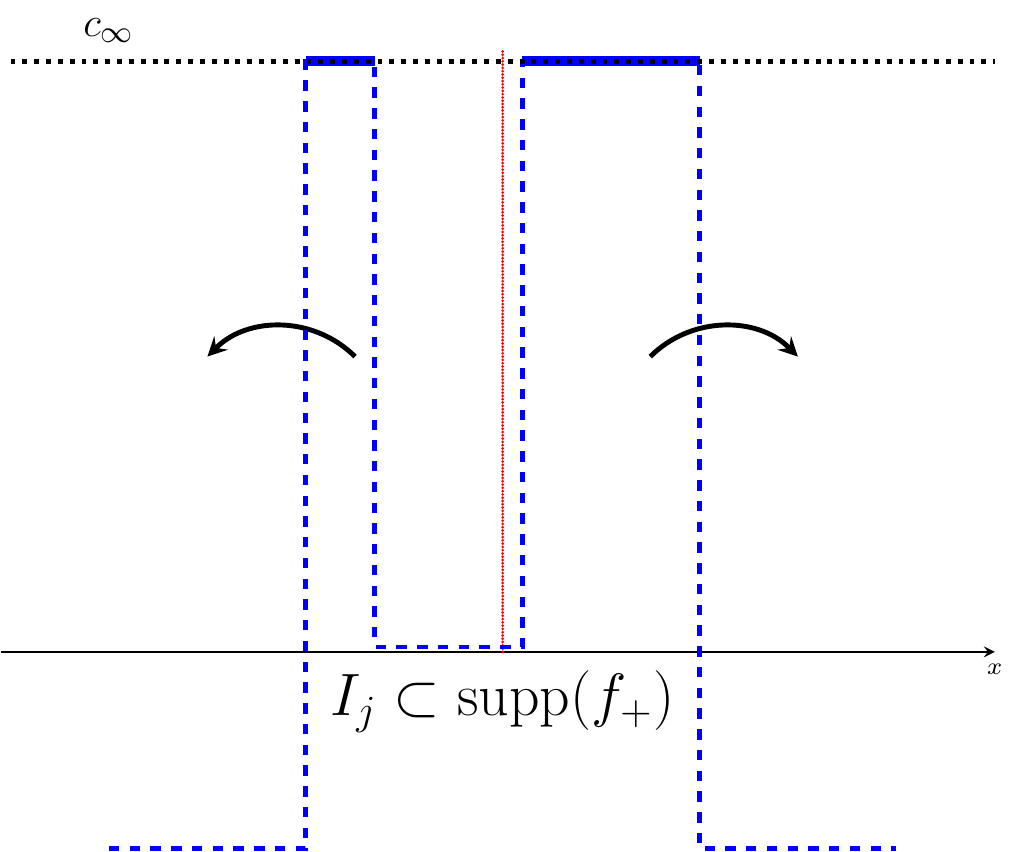}
            \label{fig:post_opt_int}
        \end{subfigure}
\caption{
{\bf (A)} For an interval $I_j$ where $f(x)\geq 0$ and a given transport plan $\gamma\in \Gamma (f_+, f_-)$, the mass
%the $L^1$ mass
left of the dotted lines will be transported to intervals $I_{\ell}$ with $\ell <j$, and the mass to the right of the dotted to intervals $I_{\ell}$ with $\ell >j$. {\bf (B)} The cost of transport for the portrayed re-organization of $f$ on $I_j$ is cheaper.}
\label{fig:opt_int}
\end{figure}

To prove the above intuition rigorously, we will inductively define a sequence of functions in the following way: let $f^0 \equiv f$ and $\gamma^0 \equiv \gamma$. For $1\leq j \leq N$, we will define a function $f^j$ and a new transport map $\gamma^j \in \Gamma (f_{+}^j,f_{-}^j)$ such that they satisfy conditions (i--iii) of Lemma \ref{lem:steps} on the intervals $I_0, \ldots , I_{j-1}$ and such that $K_p(\gamma_{j-1}) \geq  K_p(\gamma_j)$ for any $p\geq 1$.

Since $f$ has a definite sign on each interval, suppose without loss of generality that $f$ is non-negative on $I_j$ (the negative case is completely symmetric). Define $p^j \equiv \gamma^j(I_j\times \cdot )$; this is a Borel measure on $I$ which specifies how much mass from $I_j$ is transported %to a certain Borel set
in $I$ under $\gamma^j$.

Since the functions $f_+$ and $f_-$ (thus the functions $f_{+}^j$ and $f_{-}^j$) have disjoint supports, $\gamma^j$ does not transports from $I^j$ into itself, and so $p^j (I_j) = \gamma^j (I_j\times I_j) = 0$. There exist two constants $\lambda_j, \varrho _j \geq 0$, such that {\rev (identifying $z_0=0$ and $z_{N+1}=L$)} $$ {\rev p^j} ((z_0, z_j)) = \lambda _j \, ,\qquad {\rev p^j} ((z_{j+1},z_{N+1})) = \varrho_j \, .$$  These are the masses transported from $I_j$ to its left and right, respectively. 

We construct a new function $f^{j+1}\in X$ with $Z(f^j)=Z(f^{j+1})$ as follows:\footnote{Since the nodal set $Z(f^j)$ is independent of $j$, we write $z_j$ unambiguously.}
$$f^{j+1}(x) \equiv {\rm sign}(f)(x) \, \cdot \, \left\{\begin{array}{ll}
c_{\infty}  \, , & x\in \left( z_j, z_j +\frac{\lambda_j}{c_{\infty}}\right) \, , \\
0  \, , & x\in \left( z_j +\frac{\lambda_j}{c_{\infty}}, z_{j+1}-\frac{\varrho_j}{c_{\infty}}\right) \, , \\
c_{\infty}  \, , & x\in \left( z_{j+1}-\frac{\varrho_j}{c_{\infty}}, z_j \right) \\
f^j (x) \, , & {\rm otherwise} \, .
\end{array} \right.  $$
Clearly $f^{j+1}$ satisfies conditions ({\it i})--({\it ii}) of Lemma \ref{lem:steps} on the interval $I_j$, and by induction on the intervals $I_0, \ldots , I_{j-1}$ as well. Since the mass transported from $I_j$ to the left is now concentrated on $(z_j, z_j + \lambda _j / c_{\infty})$ as much as possible (with density $=c_{\infty}$), one can transport this mass to the left at a lower cost, by definition \eqref{eq:wasdef}. The same holds for the transport out of $I_j$ to the right. The transport from any other positive interval is defined identically to $\gamma^j$. The resulting $\gamma^{j+1}$, the optimal transport plan between $f^{j+1}_+$ and $f^{j+1}_-$, satisfies $K_p(\gamma^{j})\geq K_p(\gamma^{j+1})$ for any $p\geq 1$. This is because, under the new optimal transport plan, some of the mass is transported over a shorter distance, and no mass is transported over a longer distance.

Moreover, note that if $f^{j+1} \neq f^j$, i.e., if the construction really did change the function (and so also $\gamma ^j \neq \gamma ^{j+1}$), then a nonzero mass is now transported over a shorter distance, and so a strict inequality holds $K_p (\gamma ^j) > K_p (\gamma ^{j+1})$.  

Finally, we set $g\equiv f^{N+1}$.
\end{proof}

Lemma \ref{lem:steps} implies that for any $p\geq 1$
$$\min\limits_{f\in X(c_{\infty},c_1,N,I)} W_p(f_+,f_-) \geq \min\limits_{f\in X_s(c_{\infty},c_1,N, I)} W_p(f_+,f_-)  \,.$$
And furthermore the minimum on the left hand side
is attained only on $X_s(c_{\infty},c_1,N, I)$. Hence, Question \ref{q:minim} reduces as follows:

\begin{corollary}
\label{cor:equiv}
For $I=(0,L)$, the two minimization problems 
$$\min\limits_{f\in X(c_{\infty},c_1,N,I)}  W_p(f_+,f_-)
\quad\text{and}\quad \min\limits_{f\in X_s(c_{\infty},c_1,N, I)} W_p(f_+,f_-) \, ,$$
have the same minimizers and the same minimum value for any $p\geq 1$.
\end{corollary}

%Hence, a map $T$ does not transport mass between non adjacent intervals simply if .
Corollary \ref{cor:equiv} allows us consider the minimizers in Question 
\ref{q:minim} only from $X_s$, simply by shifting the mass without changing the nodal set itself. To further reduce the problem, we will now allow for {\rev non-local shifts:  we will move mass across sub-intervals,} which may also change the nodal set (but not its size). To ensure that these non-local shifts reduce the transport cost, we will make use of the monotonicity properties of \eqref{eq:Tmono}.
\begin{lemma}\label{lem:adj} 
Let $I=(0,L)$ and $f\in X_s(c_{\infty},c_1,N, I)$. There exists a function $g\in X_s$ such that $W_{p}(f_+, f_-) \geq W_{p} (g_+, g_-)$, with the following property: 
denoting the optimal transport map~\eqref{eq:Tmono} of~$g$ by $T=T[g]$, then $T$ only transport mass between adjacent intervals, i.e., $T(I_j) \subseteq I_{j-1}\cup I_{j+1}$ for every $I_j=(z_j, z_{j+1})$ where $\{z_1, \ldots ,z_N \}= Z(g)$, $z_0=0$, and $z_{N+1}= L$.
\end{lemma}

%\begin{remark}
%What we mean by not transporting mass between non-adjacent intervals is as follows: let $Z(g) = \{z_1, \ldots , z_N\} \subset (0,L)$ and add $z_0= 0$ and $z_{N+1}=L$. Then $\pi ((z_i,z_{i+1})\times (z_j,z_{j+1})) = 0$ unless $i+1 = j$ or $j+1=i$, i.e., unless the two intervals are adjacent. 

%It will be convenient here to recall that the optimal transport plan $\pi \in \Gamma(f_+,f_-)$ for $p=1$ is induced by an optimal transport map, i.e., there exists $T:{\rm supp}(f_+) \to {\rm supp}(f_-)$ with $T_{\#} f_+ = f_-$ for which $W_1(f_+,f_-) = K(({\rm id}, T)_{\#}f_+)$, see \cite{santa2015optimal} for details. Hence, a map $T$ does not transport mass between non adjacent intervals simply if $T(I_j) \subseteq I_{j-1}\cup I_{j+1}$.
%\end{remark}

\begin{proof}
Suppose without loss of generality that $f$ is nonnegative on $I_0= (z_0, z_1)$. 
%and that $N$ is even. 
By definition, $f$ is characterized by $3N$ nonnegative numbers, $\{z_i, l_i, r_i\}_{i=1}^N$, such that $f = (-1)^{i+1} c_{\infty}$ on $(z_i-l_i, z_i)$, $f= (-1)^i c_{\infty}$ on $(z_i, z_i+r_i)$, and $f=0$ everywhere else,\footnote{\rev Since $f\in X_s$, behaviors such as the interval $[F,G]$ in Figure \ref{fig:zero_nz} are excluded.} see Fig.\ \ref{fig:zlr}.
\begin{figure}[h]
\centering
{\includegraphics[scale=1]{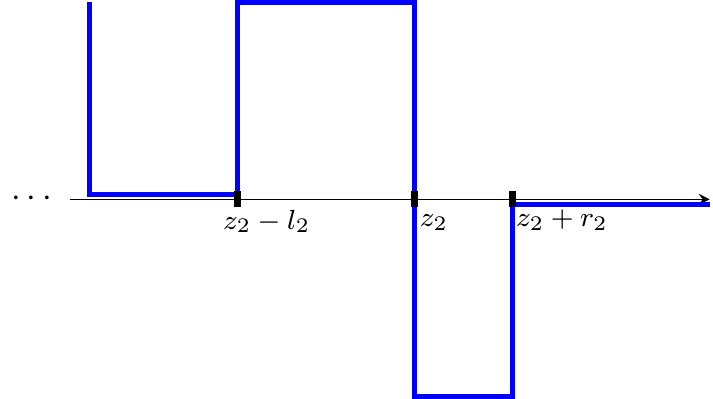}}
\caption{
A nodal point $z_2$ and the supports of $f_+$ and $f_-$ to its left and right, with $z_2-l_2$ and $z_2+r_2$ annotated. }
\label{fig:zlr}
\end{figure}

 We will again construct a sequence of functions $f^j\in X_s(c_{\infty},c_1,N)$, now characterized by $\{z_{i}^j, r_i^j, l_i^j\}_{i=1}^N$, such that
 \begin{enumerate}[label=(\roman*)]
% \item  $(T_j)_{\#}f_{+}^j = f_{-}^j$
 \item  $K_p(T^j)\geq K_p (T^{j+1})$, where $T^j$ is the monotone optimal transport map \eqref{eq:Tmono} associated with ${\rev W_p}(f^j_+, f^j_-)$.
 \item $T^{j+1}(I_{k})\subseteq I_{k+1}\cup I_{k-1}$ for $k=0,\ldots, j$  
 \end{enumerate}

 Set $f^0 = f$. Suppose without loss of generality that $f^j$ is nonnegative on $I_j$. By {\rev Lemma \ref{lem:steps},} the mass in $(z_j^j, z_j^j+r_j^j)$ is transported to the left, i.e., to $I_0, \ldots, I_{j-1}$. By the induction assumption, in the $j$-th step all of the mass in intervals with index less than $j$ is transported to adjacent intervals. Hence our inductive construction only needs to consider mass transported from $I_j$ {\em to the right.}
 
 Assume $T^j=T$ transports an interval $E \subseteq (z_{j+1}-l_{j+1}^j, z_{j+1}) \subseteq I_j$ to a non-adjacent sub-interval to the right of $I_j$, i.e., $T(E) \subseteq (z_{j+3}^j, L)$ but $T(E) \cap I_{j+1} = \emptyset$.\footnote{Since $f^j$ is non-negative on $I_j$, it is also non-negative on $I_{j+2}$, and therefore the next-nearest interval on which $f$ is non-positive is $I_{j+3}=(z_{j+3}, z_{j+4})$.}

For simplicity, suppose further that $E$ maps into a single interval, i.e., $T(E) \subseteq I_{i}$ with $i > j+2$; Otherwise, if $T$ only maps a part of $E$ into $I_{i}$, one can split $E$ into different sub-intervals each mapping into a distinct interval.

 To construct $f^{j+1}$, we would like to perform a {\bf shift operation}, that is, to {\em shift $T(E)$ into $I_{j+1}$}. To make this more precise, consider first the ideal situation where $f^j = 0$ on some subset $\tilde{E} \subseteq I_{j+1}$ of equal length, i.e., $|\tilde{E}|= |E|$. In loose terms, it means that there is space for $E$ to be shifted into $I_{j+1}$, and so we would set
 \begin{equation}\label{eq:shift}
 f^{j+1} (x) = \left\{
 \begin{array}{cc}
      -c_{\infty} \, ,& x\in \tilde{E} \, ,  \\
      0 \, ,&     x\in T(E) \\
      f^j  (x) \, , & {\rm otherwise} \, .
 \end{array} \right.
 \end{equation}
In particular, in this case $Z(f^j)=Z(f^{j+1}).$

The above construction, however, might no be possible; it might be that the interval $I_{j+1}$ is already full, by which we mean that if $(z_{j+1}^j+r_{j+1}^j, z_{j+1}^j-l_{j+2}^j)\subset I_{j+1}$, the sub-interval on which $f=0$ is shorter than $T(E)$. Then the mass of $T(E)$ cannot be shifted there (and $f^{j+1}$ cannot be defined as in \eqref{eq:shift}), since $f$ is a step function of height $\pm c_{\infty}$, and so cannot exceed this value and stay in $X_s (c_{\infty},c_1,N)$. If this is the case, we need to {\em push some or all of the points} $z_{j+2}^j, \ldots, z_{i}^j$ to the right by up to $|E|$, and then we can repeat the above construction as in \eqref{eq:shift}, with the shifted intervals and nodal points.

The shift operation is depicted on Figure~\ref{fig:shift}.
Let us consider the effect of the shift operation on the overall transport cost $W_p (f_+^{j+1}, f_-^{j+1}$):
\begin{itemize}
    \item The mass $|E|c_{\infty}$, which was previously transported between from $E\subseteq I_j$ to $T^j(E)\subseteq I_i$, is now transported over as shorter distance, to $T^{j+1}(E)\subseteq I_{j+1}$.
    \item Suppose that for $\ell>j$, the nodal endpoint $z_{\ell}^j$ was pushed to the right. By the inductive construction, the transport from/to $I_{\ell}$ could not have been from a point {\em to the left} of $z_j$. If it was transported to/from a point {\em to the right} of $I_i$, then the overall transport distance decreased.
    \item Suppose that for $\ell>j$, the nodal endpoint $z_{\ell}^j$ was pushed to the right. Suppose without loss of generality that $I_{\ell} \subseteq {\rm supp}(f_+)$ (the negative case is analogous) and that for some $D\subseteq I_{\ell}$, we have $T^j(D)\subseteq I_m \subset {\rm supp}(f_-)$ with $\ell > m >j$; the proposed shift would then {\em increase} the transport distance, by potentially pushing $D$ away from $T(D)$. This scenario, however, is {\em impossible} due to the monotonicity of $T^j$, see \eqref{eq:Tmono}: take two points $x\in E$ and $x'\in D$, then $x<x'$ but we $T(x)>T(x')$ (since $i>\ell)$, hence a contradiction.
    
    %recall that the optimal map $T^j$ is monotonically increasing (see \eqref{eq:Tmono}). Hence, on the one hand, $T^j(D)> T^j(E)$ since $D>E$ (in the sense of     inequality between every pair of points).  On the other hand}
\end{itemize} 

We note here that our construction cannot change the order of points, i.e., $z_k^{j+1}\leq z_{k+1}^{j+1}$. Similarly, the construction keeps the nodal points inside $I$, i.e., $z_1^{j+1} \geq 0$ and $z_N^{j+1}\leq L$. %So far we treated an sub-interval $E$ on the left end of $I_j$ which is transferred to the left. We proceed to treat in an analogous way sub-intervals on the right end of $I_j$ which are transferred to the right.

\begin{figure}[h]
\centering
{\includegraphics[scale=1]{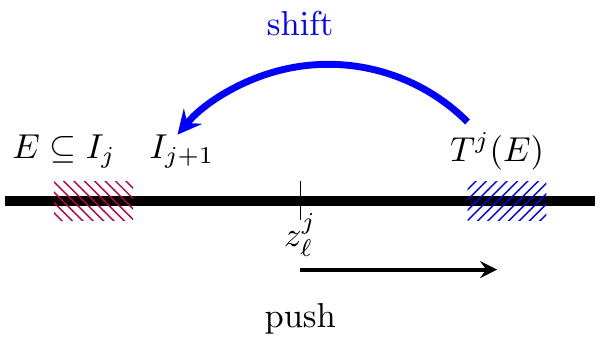}}
\caption{
The shift operation described in the proof of Lemma \ref{lem:adj}: on the $j$-th step, some $E\subseteq I_j$ is mapped by $T^j$ to a non-adjacent interval $I_i$ right of $E$. We shift that mass to $I_{j+1}$. In the process, some or all of the nodal points in between, e.g., $z_{\ell}^j$ with $j+1<\ell < i$, need to be pushed {\em away} (to the right) to ``make room'' for the shifted mass.}
\label{fig:shift}
\end{figure}

We established the first property in the induction, that $K_p(T^j)\geq K_p (T^{j+1})$. Now, note that the new transport map we constructed is monotone and pushes $f_+^{j+1}$ to $f_- ^{j+1}$. Hence, it is an optimal transport map between $f_+^{j+1}$ and $f_- ^{j+1}$ (the unique one for $p>1$), $T^{j+1}$. By construction, it satisfies the adjacency requirement, that $T^{j+1}(I_k) \subseteq I_{k+1}\cup I_{k-1}$ for all $k\leq j$.

This completes the $j$-th step. We take $g=f^{N}$, which completes the construction. Finally, note that unless $f=g$, we strictly reduced the transport cost at some stage of the induction, and so $W_p(f_+,f_-)>W_p (g_+,g_-)$.
\end{proof}

\begin{remark*}
The proof of Lemma \ref{lem:adj} relies on the convexity of the cost function $h(|x-y|)=|x-y|^p$ with $p\in [1,\infty)$. When $h$ is convex, the monotone map \eqref{eq:Tmono} is a solution of Monge's problem \eqref{eq:monge}, see \cite{santa2015optimal}. This is no longer the case if one considers a concave cost function, e.g., $|x-y|^p$ with $p\in (0,1)$. For concave costs, it is known that the maps may not be monotone \cite{gangbo1996geometry, mccann1999exact}, and a different type of analysis will be needed. Another interesting cost function not considered in this work is the $L^{\infty}$ cost \cite{barron2017duality, champion2008wasserstein}.
\end{remark*}

We conclude that $f\in X_s(c_{\infty}, c_1, N, I)$ can be a minimizer of the problem \eqref{eq:minprob}
for $p\geq 1$ if and only if it has the following structure:
\begin{enumerate}[label=(\roman*)]
\item $Z(f) = \{z_1, \ldots, z_N\} \subset (0,L)$.
\item There exist $d_1, \ldots,  d_N > 0$ such that $z_j + d_j \leq z_{j+1}-d_{j+1} $ for all $j=1,\ldots , N$, $z_1-d_1 \geq 0$, and $z_N +d_N \leq L$.
\item without loss of generality, assume that $f \geq 0$ on $(z_0, z_1)$. Then for $j$ odd, $g(x)= c_{\infty}$ on $(z_j-d_j, z_j)$ and $g(x)=-c_{\infty}$ on $(z_j, z_j+d_j)$, and vice versa for $j$ even.
\end{enumerate} 
In less formal language, a
minimizing step-function attains its maximal value $c_{\infty}$ anti-symmetrically around each nodal point $z_j$, see Figure \ref{fig:steps_opt}. The optimal transport plan $\gamma\in \Gamma (f_+, f_-)$, which is given by the monotone map \eqref{eq:Tmono}, only transports across each nodal point, and $K_p(\gamma)$ is just the sum of costs accrued at each nodal point. The only questions that now remain concern the distribution of the width parameters $d_1, \ldots d_N$, and the overall minimal optimal transport cost.
\begin{theorem}\label{thm:w1_min}
Let $I=(0,L)$ with $L>0$, $p\geq 1$, $c_{\infty}>0$, $c_1\in (0, Lc_{\infty}]$, and $N\in \mathbb{N}_{+}$.
The minimizers of the problem \eqref{eq:minprob} are step functions, anti-symmetric about each nodal point, with value $\pm c_{\infty}$ and width $c_1/(2c_{\infty}N)$. These minimizers satisfy
\begin{equation}
\min\limits_{f \in X(c_{\infty},c_1,N, I)} W_p(f_+, f_-)= 2^{-\frac{p+1}{p}} \frac{c_1}{Nc_{\infty}}  c_1^{\frac{1}{p}} \, .
\end{equation}
\end{theorem}
\begin{proof}
By \eqref{eq:wp_icdf}, the transport cost across a single nodal point $z_j$ depends on the inverse CDFs.\footnote{Since $F_{\mu}(y) = \mu( -\infty, y)$ might not be bijective, we define
{\em the inverse CDF by}
$F^{-1}(x)\equiv \inf \{y \in \R ~~ |~~ F(y) \geq y \}$, see \cite[Section 2.1]{santa2015optimal}.} {\rev By Lemmas \ref{lem:steps} and \ref{lem:adj}, it suffices to consider step functions $f\in X_s$ where mass is transported only to adjacent sub-intervals.}  Suppose without loss of generality that $f=c_{\infty}$ on $(z_1-d_1,z_1)$ and $f=-c_{\infty}$ on $(z_1, z_1+d_1)$. By definition, $F_+(z_1)=F_-(z_1+d_1) = c_{\infty}d_1$. Hence $F_{f_+}^{-1}(c_{\infty}d_1)-F_{f_-}^{-1}(c_{\infty}d_1) = d_1$. On the interval (of cumulative probabilities) $(0, c_{\infty}d_1)$, the inverse CDFs are linear with slope $1/c_{\infty}$, and so for every $t\in (0, c_{\infty}d_1) $ the difference between the inverse CDFs is constant, i.e., $d_1$. Hence, 
\begin{equation}\label{eq:singleBump}
\left[\int\limits_{0}^{c_{\infty}d_1} |F_{f_+}^{-1} (t) -F_{f_-}^{-1} (t)|^p \, dt\right]^{\frac{1}{p}}  = [c_{\infty}d_1^{p+1}]^{\frac{1}{p}} = c_{\infty}^{\frac{1}{p}}d_1^{1+\frac{1}{p}} \, . 
\end{equation}
Summing up the contributions of all nodal points, then by \eqref{eq:wp_icdf} we have that 
\begin{subequations}
\begin{equation}\label{eq:wp_djs}
W_p^p(f_+, f_-) = \sum\limits_{j=1}^{N} c_{\infty}d_j^{p+1} \, .
\end{equation}
For simplicity of computations, we will minimize $W_p^p (f_+ , f_-)$ (which is equivalent to minimizing $W_p(f_+,f_-)$). This is a constrained minimization problem, under the $L^1$ constraint
\begin{equation}
    \sum\limits_{j=1}^N d_j = \frac{c_1}{2c_{\infty}} \, .
\end{equation}
\end{subequations}
By the method of Lagrange multipliers, let $$\mathcal{L} = c_{\infty} \sum_j d_j^{p+1} - \lambda (\sum_j d_j - c_1/2c_{\infty}) \, .$$
We 
get from the condition $\partial_{d_j}\mathcal{L}=0$ that
\begin{equation}%\label{eq:dj_cond}
%\partial_{d_j}\mathcal{L}=0 ~~ \Longrightarrow    ~
\lambda = (p+1)c_{\infty} d_j^p  \, , \qquad 1\leq j \leq N\, .
\end{equation}
The condition $\partial_{\lambda}\mathcal{L} =0$ yields
$$\frac{c_1}{2c_{\infty}} = \sum\limits_{j=1}^N d_j = N\left[\frac{\lambda}{(p+1)c_{\infty}}\right]^{\frac{1}{p}},$$
which leads to,
$$\lambda = (p+1)c_{\infty} \left( \frac{c_1}{2Nc_{\infty}} \right)^p\, .$$
Combining the above expressions of $\lambda$ together, we get
$d_j = c_1/2Nc_{\infty}$, and the overall cost is, by~\eqref{eq:wp_djs},
$$
W_p^p (f_+,f_-) = Nc_{\infty}\left(\frac{c_1}{2Nc_{\infty}}\right)^{p+1} = \frac{1}{2^{p+1}} \left(\frac{c_1}{Nc_{\infty}} \right)^p c_1 \, .
$$
\end{proof}

\noindent
As a consequence, we have
\begin{corollary}
\label{cor:inequalpgeneral}
Let $I=(0,L)$ and $p\geq 1$.
For any $f \in L^\infty(I)$
we have
\begin{equation}
\label{eq:uncertnew}
   W_p(f_+, f_-)  \cdot |Z(f)|\geq 2^{-1-\frac1p} \left(\frac{\|f\|_1}{\|f\|_{\infty}} \right) \|f\|_1^{\frac1p} \, .
\end{equation}
This inequality is sharp, where equality holds if and only if $f$ is a minimizer of \eqref{eq:minprob}  as described in Theorem \ref{thm:w1_min}.
\end{corollary}

We thus see that for the case of $I=(0,L)$, we have a sharp  inequality \eqref{eq:uncertnew} in a multiplicative form. The inequality shows no dependence on the length of $L$, which can also be seen from the scaling properties of the quantities involves.
Moreover, we are able to characterize an explicit constant factor which is also the best possible. Our result establishes that for $p=1$, the inequality in \cite[Prop.\ 3.1]{cavalletti2021indeterminacy} is \eqref{eq:uncertnew}, and therefore it is indeed sharp. For $p\geq 1$, the scaling in \cite[Corollary 3.3]{deponti2022indeterminacy} of the lower bound is the same as here, $2^{-1-1/p}$, but the overall constant is lower since it is proven in more general settings. We may attribute the proof of the inequality \eqref{eq:uncertnew} not only to properties of the special geometry in one dimension, but also the way the minimization problem~\eqref{eq:minprob} is posed; it is the solution to the latter that leads to, as a by-product, the most natural ''uncertainty principle'' in a multiplicative form.

\section{The Circle}\label{sec:circle}
We now consider the minimizers of $W_1(f_+, f_-)$ in $X= X(c_{\infty}, c_1, N, I)$ for the case where $I$ is a circle. As we see from the earlier discussion on the case of an interval, by scaling of $c_1$ we can restrict our attention to the unit circle $I=S^1$. The geodesic distance between two points $e^{it},e^{is}\in S^1$ is defined by 
$$d(e^{it} ,e^{is}) = \min\{ (t-s)~ {\rm mod} (2\pi), (s-t)~ {\rm mod}(2\pi) \} \,.$$
For every $z_j \in Z(f)$ we denote $z_j =e^{is_j}$ such that, without loss of generality $$0= s_1 <s_2 < \cdots <s_N <  s_{N+1} =2\pi \, ,$$ i.e., the nodal points are ordered from the $x$-axis in a counterclockwise direction. %By a slight abuse of notations, we denote the arcs between points by $I_j =(s_j,s_{j+1})=\{ e^{it} ~~|~~ s\in(s_j,s_{j+1}) \}$, and $I_0 = (s_N, s_0)$ be the complementing arc. 

In this section we show that, on a circle, the minimizers and minimal value of \eqref{eq:minprob} are analogous to those on the interval (Theorem \ref{thm:w1_min} and Corollary \ref{cor:inequalpgeneral}):
\begin{theorem}\label{thm:wpcir}
Let $p \geq  1$, $c_{\infty}>0$, $c_1\in (0, 2 c_{\infty} \pi]$, and $N\in \mathbb{N}_{+}$. The minimizers of $W_p(f_+, f_-)$ over  $X(c_{\infty},c_1,N,S^1)$ are step functions in $X_s(c_{\infty},c_1,N,S^1)$, anti-symmetric about each nodal point, with value $\pm c_{\infty}$ and width $c_1/2c_{\infty}N$. Hence
$$\min\limits_{f \in X(c_{\infty},c_1,N, S^1)}  W_p(f_+, f_-)= 2^{-1-\frac{1}{p}} \frac{c_1}{Nc_{\infty}}  c_1^{\frac{1}{p}} \, .
$$
\end{theorem}

\begin{proof}

Throughout this proof, we use the existence of an optimal transport {\em map}; for all orders $p\geq 1$ there exists such a map $T:S^1 \to S^1$, as established for $p>1$ in see \cite{mccann2001polar} or \cite[Theorem 2.47]{villani2003topics}, and for $p=1$ in \cite{feldman2002monge} (see also \cite[Section 3.1]{santa2015optimal} and \cite{caffarelli2002constructing, bianchini2013monge} for general metric settings). A more elaborate analysis of optimal transport maps on the circle appears in \cite{delon2010fast}.

%For brevity, for all $1\leq j \leq N$, we denote the arc $\{e^{is} ~ |~ s\in (t, t')\}$ by $(t,t')$, and in particular the arcs between the nodal points $z_j = e^{is_j}$ are denoted by $I_j = (s_j, s_{j+1})$ for all $1\leq j \leq N$. Suppose without loss of generality that $f$ is non-negative on $I_1$; it can always be arranged that way, up to a rotation of $f$ which does not influence the size of $Z(f)$ or $W(f_+, f_-)$.

First, we note that the proof of Lemma \ref{lem:steps} carries to the circle without change: it is an iterative process done on $f$ in each interval (arch) $\{e^{is} ~ |~ s\in (s_j, s_{j+1})\}_{j=1}^{N}$. Hence, we can restrict our attention to solving the optimal-transport minimization problem on $X_s (c_{\infty} , c_1, N, S^1)$, see Definition~\ref{def:Xs}. %Our goal is therefore to generalize Lemma \ref{lem:adj} to the case of the circle.

 %As in the case of the interval, every $f\in X_s$ is characterized by $3N$ numbers, $\{s_i, l_i, r_i\}_{i=1}^N$, such that $f = (-1)^{i+1} c_{\infty}$ on $(s_i-l_i, s_i)$, $f= (-1)^i c_{\infty}$ on $(s_i, s_i+r_i)$, and $f=0$ everywhere else.

Next, we turn to extend Lemma \ref{lem:adj} to the circle, i.e., to show that a function $f\in X_s$ has a cheaper transport cost if the optimal transport map associated with it only transports mass between adjacent arcs. Here lies the main new challenge: we cannot simply implement our inductive ``shifting'' strategy from Lemma \ref{lem:adj}, since there are no end-points to the circle, and hence no natural candidate arc $I_j$ to start the induction from. To this end, we prove the following lemma: 

\begin{lemma}\label{lem:circle_mono}
Let $f\in X_s$ and let $T$ be the optimal transport map associated with $W_p (f_+, f_-)$ for a fixed $p \geq 1$. Then, there exists a partition $S^1 = J_1 \cup \cdots \cup J_{K}$ into disjoint arcs such that for each arc $J_k$, either
 \begin{enumerate}
     \item All points $x\in J_k \cap {\rm supp}(f_+)$ are transported clockwise to $T(x) \in J_k$, {\em or}
     \item all points $x\in J_k \cap {\rm supp}(f_+)$ are transported counterclockwise to $T(x) \in J_k$, {\em or}
     \item $f= 0 $ on  $J_k$.
     \end{enumerate}
\end{lemma}
{\rev To prove Lemma \ref{lem:circle_mono}, we will rely on the existence of an optimal transport {\em map} (in the sense of Monge). For $p>1$, recall the following theorem due to McCann \cite{mccann2001polar} (for the Euclidean case, see \cite{brenier1991polar}), presented here in a simplified form:
\begin{theorem*}[McCann \cite{mccann2001polar}]
Let $\Omega$ be a $C^3$ connected, compact, Riemannian manifold without boundaries. Let $p>1$ and consider two Borel probability measures $\mu$ and $\nu$ with finite $p$-moments, such that $\mu$ is absolutely continuous with respect to the volume measure of $\Omega$. Then, with respect to the Wasserstein-$p$ distance \eqref{eq:wasdef}, there exists an optimal transport {\em map} $T$. Moreover, there exists a vector field $V$ such that 
\begin{equation}\label{eq:mccann}
T (x) =\exp_x [V] \, ,
\end{equation}
where $\exp$ is the exponential map with respect to the geodesic distance.
\end{theorem*}
\begin{remark*}
McCann's theorem holds for the general class of optimal transport problems with respect to $K_c$ (see \eqref{eq:Kc}) for a strictly convex cost function $c(x,y)=h(|x-y|)$. Furthermore, the vector-field $V$ is characterized in terms of the gradient of the Kantorovich potential, see \cite[Theorem 13]{mccann2001polar} for details.
\end{remark*}
The case of $p=1$ is similar: the circle decomposes into a union of geodesics lines (arcs), known as ``transport rays'', which intersect (potentially) only at their endpoints. On each such transport ray, the optimal transport map is monotone; see details in \cite[Section 3.1]{santa2015optimal} and \cite{feldman2002monge}.}
\begin{proof}[Proof of Lemma \ref{lem:circle_mono}]
First consider $p>1$. Choose any $x\in S^1$, and suppose first that $y=T(x)$ is transported clockwise from $x$, i.e., $y$ is clockwise to $x$ on the shorter arc between the two points. Then, any point $w\in S^1 \cap {\rm supp} (f_+)$ lying on that arc is also transported clockwise, since the vector field $V$ points clockwise on that arc. Hence, the set of points $x\in S^1$ for which $T$ transports clockwise is a union of arcs, and so each $J_{k}$ is a connected component of that set. If $y=T(x)$ is counterclockwise, the proof is analogous, and together these type of arcs cover ${\rm supp}(f_+)\cup {\rm supp}(f_-)$. The remaining points on $S^1$ are those where $f=0$, and by the extension of Lemma \ref{lem:steps} to the circle, it too is covered by disjoint arcs. This completes the proof for $p>1$. 

{\rev For $p=1$, the decomposition of the circle into transport rays, on each of which the optimal-transport map is monotone, yields an analogous proof.}
\end{proof}

{\rev {\it Proof of Theorem \ref{thm:wpcir} - continued:}} Suppose first that $J_1 \neq S^1$, i.e., it is not the case that all points $x\in {\rm supp}(f_+)$ are transported clockwise (or counterclockwise). Then, on each arc $J_k$ we can apply the analysis from the case of the interval. Suppose points are transported clockwise on $J_k$. Therefore the counterclockwise end of $J_k$ has to be in ${\rm supp} (f_+)$, and we can choose it as the starting point of the inductive process in Lemma \ref{lem:adj}.

 \begin{figure}
\centering
{\includegraphics[scale=1]{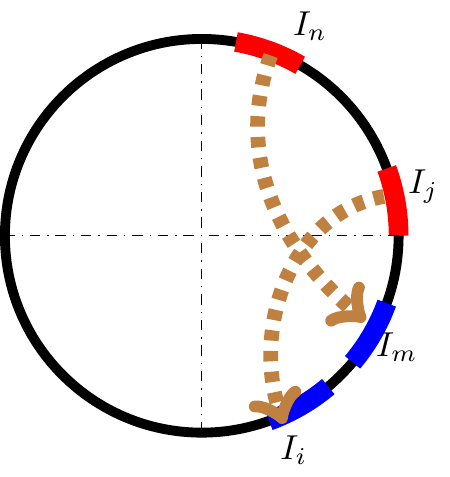}} 
\caption{
The ``lagging'' scenario as defined in the proof of Lemma \ref{lem:circle_mono}: red intervals are in the support of $f_+$, Blue intervals are in the support of $f_-$. The two dash-dotted lines are equator lines: They indicate that e.g., $I_j$ would transport to $I_i$ counterclockwise {\rev under} $T$.}
\label{fig:lagging}
\end{figure}

If, on the contrary, $J_1 = S^1$ assume that without loss of generality, all points $x\in {\rm supp}(f_+)$ are transported counterclockwise. If each $x\in {\rm supp}(f_+)$ is transported to the adjacent interval, the proof is completed. Assume otherwise, i.e., that we are in a {\bf ``lagging'' scenario} as in Figure \ref{fig:lagging}. We claim that this cannot be the optimal transport map for a minimizer of $W_p (f_+, f_-)$. This can be shown by constructing $g\in X_s$ for which $W_p(g_+,g_-)<W_p (f_+, f_-)$ as follows:

Let $Z(f) = \{z_1, \ldots z_N\}$ be the nodal points arranged in a counterclockwise order, with an arbitrary starting point. As before, denote by $I_j$ the arcs between $z_j$ and $z_{j+1}$, where $I_N$ is the arc between $z_N$ and $z_1$. Assume without loss of generality that $I_1 \subseteq {\rm supp}(f_+)$, and set $I_1$ for $g$ to be the same as for $f$. Adjacent to $I_1$ counterclockwise we set a negative arc $I_2$ precisely of the length $|I_1|$, and we define the new pushforward map $S_{\#}g_+ = g_-$ as the monotonic map from $I_1$ to $I_2$. We do so iteratively - we position positive intervals precisely of the size they had for $f$, and then a negative interval of the same size. In defining $g$, the $L^1$ and $L^{\infty}$ norms are unchanged, and so it the number of nodal points. While the map $S$ might not be the optimal transport map between $g_+$ to $g_-$, the overall transport distances have been reduced and so $W_p(f_+, f_-)=K_p(T)> K_p (S)>W_p (g_+, g_-)$.

Hence, we have extended Lemma \ref{lem:adj} to the circle, for all $p\geq 1$. Now, since the transport occurs only across nodal points to adjacent intervals, our Lagrange-multipliers analysis for the case of the interval applies, and we obtain the desired result.
\end{proof}

\begin{remark*}
The work of Delon, Salomon, and Sobolevskii \cite{delon2010fast} suggests an alternative route to prove Lemma \ref{lem:circle_mono} on the circle:
``lifting'' each measure on the circle to a periodic measure on the line, 
they study locally optimal transport maps between the ``lifted'' periodic measure.
These measures are similar to $F_{\mu}^{-1} \circ F_{\nu}$ up to a shift, and therefore in particular, are monotonic. We do not pursue this strategy further in this work, and refer to \cite{delon2010fast} for details.
\end{remark*}

\section{Star graphs}\label{sec:trees}

Given a positive integer $D$, we define the star graph $S_D = S_D (L_1, \ldots, L_D)$ as the quotient space of the disjoint union of $D$ intervals $I_j = [0,L_j)$ where $L_j>0$ for $1 ,\ldots , D$, under the equivalence relation $0_{I_j} \equiv  0$. For ease of notation, denote by $t_{j}$ the point $t\in I_j$ for every $1\leq j \leq D$ and every $t\in (0,L_j]$.

We will call the $0$ point the vertex of the star. Definition \ref{def:nodal} for $Z(f)$ extends to $I=S_D$. Consequently, the definitions of $X(c_{\infty},c_1,N,I)$ (see \eqref{eq:setX}) and
$X_s(c_{\infty},c_1,N,I)$ (see Definition \ref{def:Xs}) naturally extend to the case of star graph $I=S_D$ as well. {\rev The distance between any two points $x\in I_j$ and $y\in I_k$ is $|x-y|$ if $j=k$, and $x+y$ otherwise, i.e., if the two points are on different edges of the star, the geodesic distance between them is the length of the path going from $x$ to $y$ through the vertex $0$.}

\begin{figure}
{\includegraphics[scale=1]{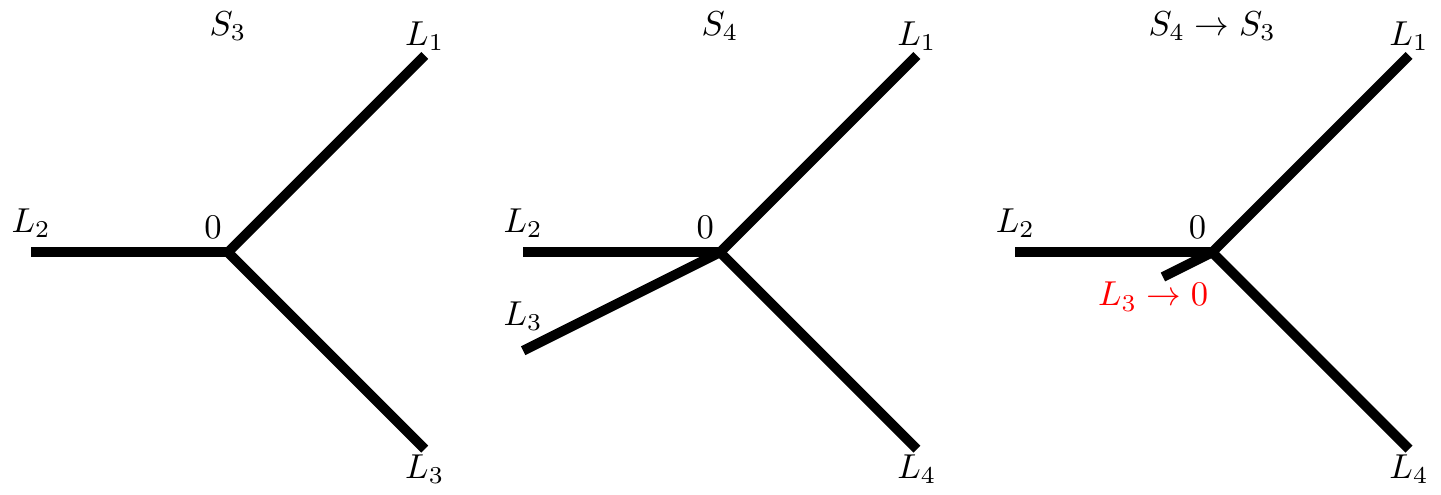}} 
\caption{Left and Center stars have $3$ and $4$ long edges, respectively. But in a star graph, the lengths of the edges matter: as the length $L_3 \to 0^+$, $S_4$ is deformed into $S_3$ (Right).
This metric structure in turn manifests itself in the minimization problem and in its solution, see Theorem \ref{thm:starequi}.}
\label{fig:stars}
\end{figure}

Stars are a class of spaces where we might expect the optimal dependence between $W_p (f_+,f_-)$ and the number of nodal points $N$ to be non-multiplicative, for the following reason: In \cite{cavalletti2021indeterminacy, de2021eigenfunctions},  the sharp (up to a constant) uncertainty quantification \eqref{eq:uncert} was derived for metric spaces which are essentially non-branching. Intuitively, this means that if $\ell_1 (t), \ell_2 (t):[0,1]\to \Omega$ are two ``generic'' geodesic lines of unit length, and $\ell_1 (t)=\ell_2(t)$ on an open subset of $[0,1]$, then $\ell_1=\ell_2$ everywhere; see \cite{rajala2014non} for details.  Star graphs, however, are certainly spaces with branching (and so are trees in general). Hence, in search for new types of dependencies between $\min W_p (f_+, f_-)$ and $N$, stars are excellent candidates over which to study the minimization problem stated in Question \ref{q:minim}.

The technical difficulty is that, on the star, we do not have explicit optimal maps such as \eqref{eq:Tmono} for the interval, nor do we even expect the existence of optimal transport maps, i.e., we expect a solution to the Kontorovich problem \eqref{eq:wasdef}, but not Monge's \eqref{eq:monge}. More broadly, stars are an example of metric graphs, on which the study of optimal transport is at a relatively early stage~\cite{erbar2021gradient, mazon2015optimal}.

{\bf Main results:} The key element of our analysis is that it is always ``useful'', in the sense of minimizing $W_p(f_+,f_-)$, to position one of the nodal points at the vertex of the star. We approach Question \ref{q:minim} on stars by establishing
a correspondence between transport over a star-vertex and transport on the real line (Lemma \ref{lem:star_vertex}). This equivalence allows us to reduce the minimization problem to a finite-dimensional constrained optimization problem (Theorem \ref{thm:starequi}). 

From there, the optimization problem bifurcates into several different cases, depending on both the topology and the lengths of the star's edges. In Sections \ref{sec:star_longN1}--\ref{sec:starD3_short}, we work the details of the following cases:
\begin{itemize}
    \item For an {\bf even number of sufficiently long edges,} the vertex $0\in S_D$ is equivalent to $%\tilde{D}=
    D/2$ nodal points. Hence, we have a multiplicative uncertainty principle of the type (presented here for simplicity with $p=1$)
    $$ W_1 (g_+,g_-) \geq \frac{c_1^2}{4c_{\infty}}\frac{1}{\tilde{N}} \, , \qquad \tilde{N}=N-1 + %\tilde{D},\quad \tilde{D}=
    \frac{D}{2} \, ,$$
    see Section \ref{sec:star_longDeven}.
    \item For an {\bf odd number of sufficiently long edges,} the main complication is that there is an imbalance between the number of positive and negative edges around $0$. Nevertheless, we get the same type of inequality, only now with (Section \ref{sec:star_longOdd})
    $$ \tilde{N}\equiv N-1 + \tilde{D} \, , \qquad \tilde{D}\equiv \frac{(D+1)(D-1)}{2D} \, .$$
    \item When {\bf one of the edges is short,} no $\tilde{N}$ or $\tilde{D}$ type inequalities emerge, but the lower bounds we obtain ``interpolate'' between the case of a star with $D-1$ long edges (and a ``degenerate'' $D$-th edge with length zero) to the case of $D$ long edges (Section \ref{sec:starD3_short}).
\end{itemize}

In summary, for stars with $D$ long edges, the vertex $0\in S_D$ is effectively equivalent to
$D/2$ or $\tilde{D}=\frac{(D+1)(D-1)}{2D}$ 
 nodal points on the line, depending on whether $D$ is even or odd, respectively.
Finally, while we can interpolate between $D$ to $D-1$ edges by shortening/lengthening the edges, the general case of a star does not seem to
admit such a clean result; Indeed, an  ``uncertainty principle''-type lower bounds on $W_1(f_+,f_-)\cdot |Z(f)|$, such as \eqref{eq:uncertnew}, breaks even in relatively simple metric graphs, thus demonstrating that the non-branching property used in \cite{cavalletti2021indeterminacy} is indeed necessary.

\subsection{Stars -- the general framework}
Our strategy to prove the main result, Theorem \ref{thm:starequi}, consists of two parts. In Lemma \ref{lem:star_vertex}, we analyze the optimal transport problem in the case of a single nodal point on the star vertex. Lemma \ref{lem:adj_sd} generalizes Lemma \ref{lem:adj} on the optimality of transfer to adjacent sub-interval to the case of the star graph. 
The main technical difficulty here is that we cannot assume the existence of optimal transport maps (in the sense of Monge's problem \eqref{eq:monge}, as opposed to Kantorovich \eqref{eq:wasdef}), on graphs. Indeed Lemma \ref{lem:star_vertex} shows that already in simple settings of $Z(f)=\{0\}$ the optimal transport plan will not be induced by a map. Moreover, we cannot expect to have monotonicity in the strict sense, due to the geometry of the graph. Nonetheless, Lemma~\ref{lem:mono_sd} shows that the optimal transport plans satisfy a sufficient monotonicity-like property.

%Our starting point are
We begin with functions satisfying $Z(f) = \{0\}$.
\begin{lemma}\label{lem:star_vertex}
Let $D >M\geq 1$ be integers and consider $ f\in X(c_{\infty}, c_1, 1, S_D)$ where for every $x_j \in I_j \subseteq S_D$,
\begin{equation}
    f(x_j) 
     \left\{ \begin{array}{ll}
\geq 0 \, , & 1\leq j\leq  M \,, \\
\leq 0 \, , & M+1\leq j \leq  D\,,
\end{array} \right. .
\end{equation}
Define $g: \R \to \R$  as\footnote{\rev Here we identify $x\in\R$ with the point $x_j\in I_j$ with the exact same value. This is unambiguous since each $f_j$ is only defined on the respective edge $I_j$.}
\begin{equation}\label{eq:gstar_def}
g(x) =g (x;f)= g_+(x) + g_-(x) \equiv  \sum\limits_{j=1}^M {\rev f_j (x)} 
+ \sum\limits_{j=M+1}^{D}  {\rev f_j(-x)} %\mathbbm{1}_{I_j}(-x) \, ,
\end{equation}
such that $g_+$ is supported on $(0,\max_{j=1,\ldots, M} L_j)$ and $g_-$ is supported on $(-\max_{j=M+1,\ldots, D} L_j,0)$. Then there is a surjective correspondence $\Phi :\Gamma( f_+, f_-) \to \Gamma(g_+, g_-)$ such that $K_p(\gamma) = K_p( \Phi[{\gamma}])$ for every $p\geq 1$ and $\gamma \in \Gamma (f_+, f_-)$. Hence, $$ W_p(g_+,g_-) = W_p(f_+,f_-) \, .$$
\end{lemma}
\begin{proof}

For for every $\gamma \in \Gamma (f_+, f_-)$ and any two measurable sets $E_+, E_- \subseteq (0,\infty)$, define 
\begin{equation}\label{eq:Phidef}
\Phi[\gamma] (E_+, -E_-) \equiv \sum\limits_{j=1}^{M}\sum\limits_{i=M+1}^D \gamma(E_+ \cap I_j, -E_- \cap - I_i) \, ,    
\end{equation}
 where for every $A \subseteq (0,\infty)$ we define $-A \equiv \{-x ~~|~~ x\in A \}$. $\Phi[\gamma]$ is a Borel measure on $\R \times \R$ with  marginals $g_+$ and $g_-$, simply by additivity. That $K_p(\gamma) =~K_p(\Phi[\gamma])$ follows from definition \eqref{eq:wasdef}: the same amount of mass travels the same distance in both plans. To summarize, we have shown the inclusion $$\{K_p (\gamma) ~ ~| \gamma \in \Gamma (f_+, f_-) \} = \{K_p (\Phi[\gamma]) ~ ~| \gamma \in \Gamma (f_+, f_-) \} \subseteq \{K_p (\eta ) ~~|~~\eta \in \Gamma (g_+, g_-) \}\, .$$

To conclude that $W_p(f_+, f_-)= W_p (g_+,g_-)$ we need to prove inclusion in the other direction, i.e., to find for any $\eta \in \Gamma (g_+,g_-)$ a coupling $\gamma_{\eta}\in \Gamma (f_+, f_-)$ such that $\Phi [\gamma_{\eta}]=\eta$. We define $\gamma_{\eta}$ in a symmetric manner: for any two measurable sets $A , B \subseteq (0,\infty)$ and $1\leq i,j \leq D$, 
$$\gamma_{\eta} (A\cap I_i, B \cap I_j) \equiv \eta (A, -B) \frac{\|f_i\|_{L^1(A\cap I_i)}}{\|g_+\|_{L^1(A)}} \frac{\|f_j\|_{L^1(B\cap I_j)}}{\|g_-\|_{L^1(-B})}\, ,$$
and for any general sets $\mathcal{A}, \mathcal{B} \subseteq S_D$, then 
$$\gamma_{\eta} (\mathcal{A}, \mathcal{B}) \equiv \sum\limits_{i=1}^D \sum\limits_{j=1}^{D} \gamma_{\eta} (\mathcal{A} \cap I_i, \mathcal{B}\cap I_j )\, . $$
Again, by definition $\gamma_{\eta}$ is a Borel measure on $S_D \times S_D$. To verify that $\gamma_{\eta}\in \Gamma (f_+, f_-)$, we compute its marginals. Let $(x,y) \in {\rm supp} \gamma_{\eta}$. Necessarily $y$ is an element in one of the negative intervals $I_{M+1}, \ldots I_D$, then for any measurable set $A\subseteq (0, \infty)$
\begin{align*}
    \gamma_{\eta} ( A\cap I_i, S_D) &= \gamma_{\eta} \left( A\cap I_i, \bigcup\limits_{j=M+1}^{D}I_j\right) \\
    &=\sum\limits_{j=M+1}^{D} \gamma_{\eta} ( A\cap I_i, I_j) \\
    &= \eta(A, (-\infty,0)) \frac{\|f_i\|_{L^1(A\cap A_i)}}{\|g_+\|_{L^1(A)}}
    \sum\limits_{j=M+1}^{D}  \frac{\|f_j\|_{L^1(I_j)}}{\|g_-\|_{L^1(-\infty,0)}}\\
    &= \eta(A, (-\infty,0))\frac{\|f_i\|_{L^1(A\cap I_i)}}{\|g_+\|_{L^1(A)}} \cdot 1 = \|f_i\|_{L^1(A\cap I_i)} \, ,
\end{align*}
where at the last passage we used the fact that $g_+$ is the density of the second-coordinate marginal of any $\eta \in \Gamma (g_+, g_-) $, i.e., that $\eta(A, (-\infty,0))=\|g_+\|_{L^1(A)}$. The calculation of the second-coordinate marginals of $\gamma_{\eta}$ is analogous. 

Finally, to prove that $\Phi[\gamma_{\eta}]=\eta$, consider a pair of open sets $E_+, E_- \subseteq (0,\infty)$. By \eqref{eq:Phidef}
\begin{align*}
    \Phi[\gamma_{\eta}](E_+, -E_-) &= \sum\limits_{j=1}^M\sum\limits_{i=M+1}^D \eta (E_+, -E_-)\frac{\|f_i\|_{L^1(E_+\cap I_i)}}{\|g_+\|_{L^1(E_+)}} \frac{\|f_j\|_{L^1(E_-\cap I_j)}}{\|g_-\|_{L^1(-E_-)}} \\
    &=\frac{\eta (E_+, -E_-)}{\|g_+\|_{L^1(E_+)}\cdot \|g_-\|_{L^1(-E_-})} \sum\limits_{j=1}^M \|f_j\|_{L^1(E_-\cap I_j)}\sum\limits_{i=M+1}^D \|f_i\|_{L^1(E_+\cap I_i)} \\
    &=\frac{\eta (E_+, -E_-)}{\|g_+\|_{L^1(E_+)}\cdot \|g_-\|_{L^1(-E_-)}} \|f_+\|_{L^1(E_+)} \|f_-\|_{L^1(-E_-)} \\
    &= \eta (E_+, -E_-) \, .
\end{align*}
\end{proof}

To solve the minimization problem on $S_D$, we first state an extension of Lemma \ref{lem:steps} to star graphs, where the proof is completely identical to that of Lemma~\ref{lem:steps}: 
\begin{lemma}\label{lem:steps_Sd}
For every $f\in X= X(c_{\infty},c_1,N,S_D)$, there exists a function $h \in X_s(c_{\infty}, c_1, N ,S_D)$ such that
\begin{enumerate}[label=(\roman*)]
\item $Z(f) = Z(h)$ %$ = \{z_1, \ldots, z_N\}$.
\item For any sub-interval $J$ between to adjacent nodal points, i.e., $J=(z,z')$ where $z,z'\in Z(f)$ and $(z,z')\cap Z(f) = \emptyset$,
then $h \geq 0$ on $J$ if and only if $f\geq 0$ there, and
\begin{equation}\label{eq:Jhequalmass}
    \int_{J} h(x) \,dx = \int_{J} f(x) \, dx \, .
    \end{equation}
If $0 \not\in Z(f)$, \eqref{eq:Jhequalmass} also holds when $J$ is a
maximal star-like subgraph for which $0\in J$ but $Z(f)$ is disjoint from the interior of $J$.

\item $W_p (f_+,f_-) \geq W_p (h_+, h_-)$ for any $p\geq 1$ with equality possible only if $f\in X_s(c_{\infty}, c_1, N ,S_D)$ as well.
\end{enumerate} 
\end{lemma}
%it is always preferable for $f$ to be a step function with values $\pm c_{\infty}$ adjacent to the nodal point. 
We now proceed to prove an analog of Lemma \ref{lem:adj} for star graphs; the minimizers of $W_p (f_+,f_-)$ in $X_s(c_{\infty}, c_1, N ,S_D)$ are those where mass is transported only between adjacent sub-intervals. 
\begin{lemma}\label{lem:adj_sd}
Let $f\in X_s(c_{\infty},c_1,N, S_D)$. There exists $g\in X_s(c_{\infty}, c_1, N ,S_D)$ such that $W_p(f_+, f_-) \geq W_p (g_+, g_-)$ {\rev for every $p\geq 1$}, with the following property: 

If $\gamma \in \Gamma (g_+, g_-)$ is an optimal transport plan, then $\gamma$ only transport mass between adjacent intervals. By this, we mean that if $J$ and $J'$ are two {\em closed} sub-intervals between two adjacent nodal points (or the smallest star-like subgraph between the nodal points closest to $0$), then $\gamma (J \times J') \neq 0 $ only if $\partial J \cap \partial J' \neq \emptyset$.
\end{lemma}
The new challenge in proving Lemma \ref{lem:adj_sd} is the absence of a monotonicity property. On the interval, the explicitly optimal transport map \eqref{eq:Tmono} is monotone. On the circle, monotonicity is a consequence of the exponential map form of \eqref{eq:mccann} (or of the more specific construction in \cite{delon2010fast}). On star graphs, while
we cannot hope for monotonicity in the usual sense, we show that a similar property (Lemma \ref{lem:mono_sd}) is sufficient to demonstrate
that the shift operations described in Lemma \ref{lem:adj} can be implemented on $S_D$ as well (Lemma \ref{lem:adj_sd}).

\begin{proof}[Proof of Lemma \ref{lem:adj_sd}]
We define the following disjoint partition:
$$ S_D = J^* \cup ~ \bigcup\limits_{i=1}^D \bigcup\limits_{j=1}^{n(i)} J_{i,j} \, ,$$
where on each interval $I_i$ that comprises the star graph $S_D$, we denote the sub-intervals between adjacent nodal points by $J_{i,1}, \ldots , J_{i,n_{i}} \subset I_i$ where $n_i \geq 0$ for each $1 \leq i \leq d$, and where the subintervals are ordered by {\em decreasing} distances from $0$. Finally, let $J^*$ be the smallest star-like graph between the closest nodal point to $0$. If $0\in Z(f)$, then $J^* = \emptyset$. For ease of notation, it is useful to denote $J^* = J_{i, n(i)+1}$ for all $1\leq i \leq D$.

We repeat the inductive construction of Lemma \ref{lem:adj} (see Figure \ref{fig:star_adj} for illustration): for each $j$, we iterate over all  $1\leq i \leq D$ for which $j\leq n(i)$. Suppose without loss of generality that $J_{i,j} \cap {\rm supp} (f_+) \neq \emptyset
$ and that for some set of nonzero measure $E \subseteq J_{i,j} \cap {\rm supp} (f_+)$ we have that $\gamma(E, \cdot)$ is supported on non-adjacent intervals, i.e., 
$$ c \equiv \gamma \left( E, S_D\setminus \left(J_{i,j-1} \cup J_{i,j+1} \right) \right) > 0 \, .$$
For each non-adjacent interval where $\gamma (E, \cdot)$ is supported, we will shift that exact same mass to either $J_{i,j-1}$ or $J_{i. j+ 1}$, depending on the relative order between the intervals.

\begin{figure}
\begin{subfigure}[c]{1\textwidth}
\includegraphics[scale=1]{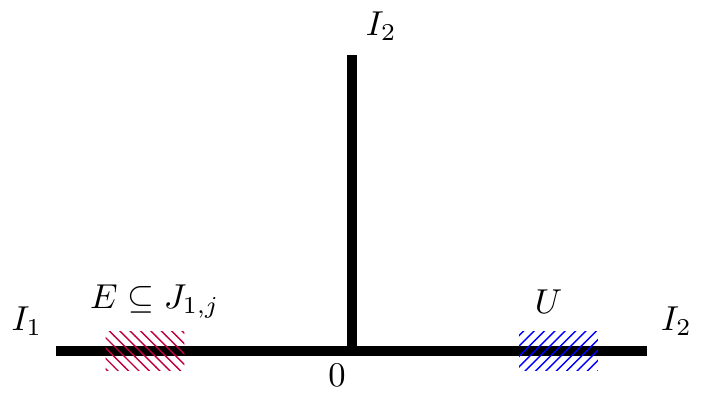}
\centering
\caption{}
\end{subfigure}
\begin{subfigure}[c]{1\textwidth}
\includegraphics[scale=1]{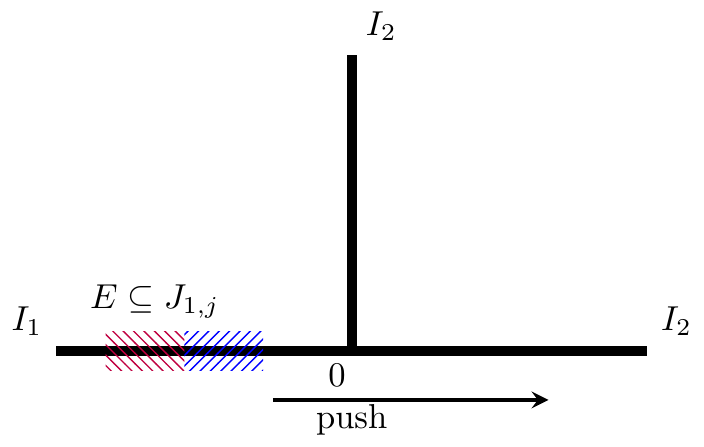}
\centering
\caption{}
\end{subfigure}
\begin{subfigure}[c]{1\textwidth}
\includegraphics[scale=1]{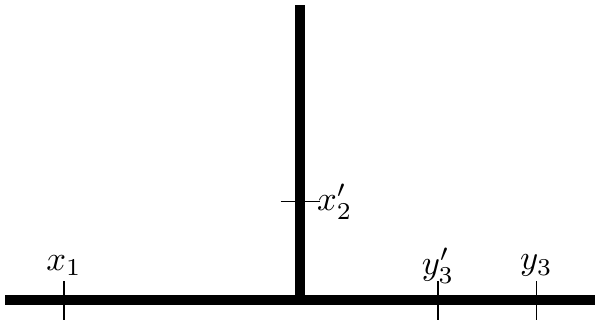}
\centering
\caption{}\label{fig:xyxpyp}
\end{subfigure}
\caption{The shifts of Lemma \ref{lem:adj_sd} on $S_3$. {\bf (A).} $E$ is being (partially) transported into $U$ under the optimal transport plan $\gamma$. {\bf (B).} After the shift, the relevant mass of $U$ is shifted to the interval adjacent to $E$, where mass on some of the intervals between $E$ and where $U$ previously was has
been shifted to the right. {\bf (C).} If $x_1\in E$, $y_3\in U$, i.e., $(x_1, y_3)\in {\rm supp}(\gamma)$, then this ``push'' could be harmful if $(x'_2,y'_3)\in {\rm supp}(\gamma)$, since then $y'_3$ could be pushed away from $x'_2$. We show that this scenario is forbidden when $\gamma$ is an optimal transport plan with respect to the cost function $h(x-y)=|x-y|^p$ for $p \geq 1$. 
}
\label{fig:star_adj}
\end{figure}

As in Lemma \ref{lem:adj}, these shifts might require to {\em push away from $J_{i,j}$} some nodal points to create room for the negative mass. See Figure \ref{fig:star_adj} for an illustration. Suppose then that some interval of length $A$ was shifted from $J_{i', j'}$ to $J_{i,j\pm 1}$, and that all or some of the mass between them has been pushed away from $J_{i, j\pm 1}$. To exclude these scenarios, we state a monotonicity lemma (in the spirit of \cite[Theorem 2.9]{santa2015optimal}):
\begin{lemma}\label{lem:mono_sd}
Let $\gamma$ be an optimal transport plan with respect to $W_p$ with {\rev $p\geq 1$,} let $1\leq i,j,k,l \leq D$, let $x,x',y,y'>0$, and let $(x_i,y_k), (x'_j,y'_l)\in {\rm supp}(\gamma)$. 
Then the following scenarios are impossible:
\begin{enumerate}
    \item $i,j,k$ are distinct, $k=l$,  $x'<x$, and $y'<y$.
    %\item $i,j,k$ are distinct and $y'<y$.
    \item $j,k,l$ are distinct, $i=j$,  $x'<x$, and $y'<y$.
    %\item $i,j,k$ are distinct and $y'<y$.
    \item $i=j$, $x'<x$, and $y'<y$ (it may or may not be that $i=k$, i.e., $(x_i,y_k)$ may or may not be part of the same edge).
\end{enumerate}
{\rev Here, as before, for every $1\leq i \leq D$ and every $x\in (0, L_i)$, we denote by $x_i$ the point on $I_i$ whose distance from the vertex $0$ is $x$.}
The same notation applies to $y_k$, $x'_j$, and $y'_l$.
\end{lemma}
 The proof of Lemma \ref{lem:mono_sd} is given after we complete the proof of Lemma \ref{lem:adj_sd}.
 
Lemma \ref{lem:mono_sd} is a star-graph analog of the monotonicity property \eqref{eq:Tmono} of optimal transport on the line: item (3) is the most straightforward, since it implies that all four points lie on the same linear path. Item (1) (and analogously (2)) is depicted in Figure \ref{fig:star_adj}(C) with $i$, $j$, and $k$ being distinct, $k=l$, i.e., 
 $x_i$ and $x'_j$ do not lie on the same edges, and $y'_k$ is on the path between $x_i$ and $y_k$.

With the help of
Lemma \ref{lem:mono_sd}, we now make the following claim: if a set with nonzero mass between $J_{i,j}$ and $J_{i',j'}$ was pushed away from $J_{i,j}$, then it cannot increase the overall transport cost. 

As in the case of the interval or the circle, proving this claim would show that our inductive construction decreases the transport cost, thus proving the lemma.

To proceed with the proof of the above claim, we note that such a ``push" away of a set $D$ from $J_{i,j}$ (towards the vertex $0$) could increase the overall transport cost in three scenarios: either {\it (i)} the optimal-transport plan $\gamma$ couples $D$ to a set {\em further away} from the vertex on $I_i$, i.e., in $J_{i,1}, \ldots, J_{i,j-1}$, or, {\it (ii)} $\gamma$ couples $D$ to an interval between $I_i$ and $D$ {\it (iii)} $\gamma$ couples $D$ to a {\em third} interval, i.e., not $I_i$ and not $I_{i'}$. We will now rule out both all these three scenarios.

First, {\it (i)} is impossible since there cannot be a point $x'_i \in J_{i,1} \cup \ldots \cup J_{i,j-2}$ which was previously transported into a pushed interval, due to the inductive construction. The only other way for the transport distance to increase with the shift is if $y'_k$ is on the path between $x_i$ and $y_k$ and that $(x_i, y_k), (x'_j,y'_k) \in {\rm supp}(\gamma)$ with $y'\neq y$.
Lemma \ref{lem:mono_sd}, item (3) similarly rules out scenario {\it (ii)}. We are left with scenario~{\it (iii)}, depicted in Figure \ref{fig:xyxpyp}. We remark that it is really this scenario that distinguishes the star graph from the interval.

Here, $i\neq j$, and so Lemma \ref{lem:mono_sd}, items (1)--(2), imply that $|x'|\geq|x|$.  Hence, up to relabeling of the edges, we have that  $(x_{1},y_{3}),(x'_{2},y'_{3})\in {\rm spt}(\gamma)$ with   $|x'|\geq |x|$, and
 $|y|>|y'|$. By a slight abuse of notation, let $x$ and $x'$ be
 the infimum of all values such that $x_1$ and $x'_2$ satisfy these hypotheses, i.e., the closest to the vertex $0$.
 
 To rule this scenario out, let us now define an auxiliary function, $\bar{f}$, which exchanges the values of $f$ on $(0,x)$ between $I_1$ and $I_2$. Formally, {\rev for every $j=1,\ldots D$ and every $t_j \in I_j$ }
$$ \bar{f} (t_j) = \left\{\begin{array}{ccc}
     f(t_2)& {\rm if} ~j=1  & {\rm and}~ t\in (0,x)\\
     f(t_1)& {\rm if} ~j=2  & {\rm and} ~t\in (0,x)\\
     f(t_j) & {\rm otherwise} &
\end{array}  \right.$$
Let us also define $\bar{\gamma}$ to be the coupling between $\bar{f_+}$ and $\bar{f_-}$ which is identical to $\gamma$, with the corresponding changes. We now proceed to make one helpful observation that $\bar{\gamma}$ 
is an optimal transport plan between $\bar{f}_+$ and $\bar{f}_-$.

Indeed, since $x_1$ is transported to $I_3$, then no other point on $(0,x_1)$ is transported to $(x_1, L_1)$. This would contradict monotonicity on the interval $I_1$ (see item (3) of Lemma \ref{lem:mono_sd}). The analogous statement holds for $x'_2$ as well. Hence, the ``exchange'' of intervals did not change the distances along which mass is transported, and so  $K_p(\gamma)=K_p(\bar{\gamma})$, and so $W_p(f_+,f_-) =K_p (\gamma) \geq W_p (\bar{f}_+, \bar{f_-})$. 

We now want to show an inequality in the other direction. 
Let $\gamma_*$ be an optimal transport plan for $\bar{f}$, then $\bar{\gamma_*}$ is a coupling of $f_+$ and $f_-$, and so $$W_p(f_+, f_-) = K_p (\gamma) = K_p(\bar{\gamma})\geq W_p(\bar{f}_+,\bar{f}_-) = K_p (\gamma_*)  =K_p (\bar{\gamma}_*) \geq W_p (f_+, f_-) \, .$$
Hence we have equalities throughout, and in particular $K_p(\bar{\gamma}) = W_p (\bar{f}_+, \bar{f_-})$, and so $\bar{\gamma}$ is an optimal transport plan from $\bar{f}_+$ to $\bar{f}_-$.

By construction $(x'_2,y'_3)\in {\rm spt}(\bar{\gamma})$ (as it were with $\gamma$), and due to the ``exchange'' now we have that $(x_2, y_3)\in {\rm spt}(\bar{\gamma})$. This violates item (3) in Lemma \ref{lem:mono_sd}. 

Thus we arrive at a contradiction. Hence $x'=x$. But now consider the situation where $(x_1. y_3), (x_2, y'_3)\in {\rm spt}(\gamma)$, with $y>y'$. By the exact same argument with
the roles of $y$ and $x$  interchanged, we arrive at a contradiction again. We have ruled out scenario {\it (iii)}, as depicted in Fig.\ \ref{fig:xyxpyp}.

To summarize, the proposed shift cannot increase the transport cost, and therefore Lemma \ref{lem:adj_sd} is proven.

\end{proof}

\begin{proof}[Proof of Lemma \ref{lem:mono_sd}]
{\rev Consider first the case of $p>1$, where $c(x,y)=h(x-y)= |x-y|^p$ is strictly convex.} First, recall that the support of $\gamma$ is $c$-cyclically monotone \cite[Theorem 1.38]{santa2015optimal}, i.e.,
\begin{equation}\label{eq:cCM}
    h(x_i-y_k)+h(x_k'-y_l') \leq h(x_i-y_l') + h(x_j'-y_k) \, .
\end{equation}
We will show in detail how the first of the three scenarios is impossible, as the others follow similarly.

Assume without loss of generality that $x_1\in I_1$, $y_3',y_3\in I_3$, and $x'_2 \in I_2$, as in Figure \ref{fig:xyxpyp}. Denote $a= |y|-|y'| > 0$. Then \eqref{eq:cCM} reads as
\begin{equation}\label{eq:h_ccm}
 h (|x|+|y'|+a) + h(|x'|+|y'|) \leq h(|x|+|y'|)+ h(|x'|+|y'|+a) \, .
 \end{equation}
 Suppose $|x'|\leq |x|$ first. Then $$|x'|+|y'| < |x'|+|y'|+a \leq  |x|+|y'|+a \, ,$$
$$ |x'|+|y'| \leq |x|+|y'| < |x|+|y'|+a \, . $$
We can express the ``sandwiched'' numbers, $|x|+|y'|$ and $|x'|+|y'|+a$, as linear interpolation between the endpoints $\alpha = |x'|+|y'|$ and $\beta = |x|+|y'|+a$,  
$$ |x'|+|y'|+a = t\alpha + (1-t) \beta \, , \qquad  |x|+|y'| = t \beta  + (1-t) \alpha \, ,$$
where 
 $$ t\equiv \frac{a}{a+|x|-|x'|} \in (0,1) \, .$$
However, $h$ is strictly convex for $p>1$, and so combined with \eqref{eq:h_ccm}, then
\begin{align*}
    h(\alpha)+h(\beta) &\leq h( t\alpha + (1-t) \beta) +h( t \beta  + (1-t) \alpha) \\
    &< th(\alpha) + (1-t)h(\beta ) + th(\beta)+ (1-t)h(\alpha) \\
    &= h(\alpha )+h(\beta) \, ,
\end{align*}
which is a contradiction.  

The case of $p=1$ is proven by approximating $h(x-y)=|x-y|$ by a sequence of strictly convex functions, see \cite{santa2015optimal}. Crucially, even though $x,y\in S_D$ and not on the real line, $h(x-y)$ is really a shorthand for $h$ operating on the geodesic distance (on the graph) between $x$ and $y$, and hence the procedure generalizes to the graph.

We have proven item (1) in the lemma. Item (2) is analogous, but with interchanging $x$ and $x'$ with $y$ and $y'$, respectively. Item (3) reduces to the standard analysis of the interval, since all four points lie on the same line.

\end{proof}

\begin{remark*}
While it is preferable to have $0\in Z(f)$, it might not be possible; A sufficient condition is to have {\em at least}
two sufficiently long edges. Up to a relabeling of the edges, we make the hypothesis that $$|I_1|, |I_2|> c_1/2c_{\infty} \, $$ i.e., all of the positive and negative mass could be allocated on $I_1$ and $I_2$, respectively, in intervals adjacent to $0$. This is a sufficient condition that avoids somewhat pathological star graphs for which $0$ cannot be efficiently utilized as a nodal point. Such graphs might resemble an interval attached to many short edges at its end, see e.g., Fig.\ \ref{fig:pathologicalStar}.
\end{remark*}

\begin{figure}
{\includegraphics[scale=1.3]{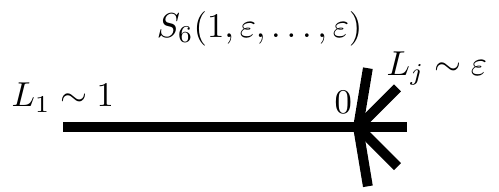}} 
\caption{A star with $6$ edges, where $L_2, \ldots ,L_6 \sim \varepsilon$ with $\epsilon<<1$. However, since most of the edges are very short, this star should be almost equivalent to an interval, in terms of transport and mass allocation. In particular, for a function with a single nodal point ($N=1$), that point cannot be the vertex $0$.}
\label{fig:pathologicalStar}
\end{figure}

 Since the minimizers of $W_p(f_+, f_-)$ on $S_D$ shift mass only between adjacent sub-intervals (Lemma \ref{lem:adj_sd}), we can now integrate our analysis of the single-point case (Lemma \ref{lem:star_vertex}) and conclude that it is always preferable to have a nodal point on the vertex. Hence, minimizers have the following form:
\begin{enumerate}
    \item The vertex is a nodal point, i.e., $0\in Z(f)$, and there exists $r_1, \ldots, r_D \geq 0 $ such that $f$ on $[0,r_j]\subseteq I_j$ is non-negative for $1\leq j \leq M$, and non-positive for $M+1\leq j \leq D$. The overall mass concentrated around $0$ is therefore
    $$\tilde{c}_1 \equiv c_{\infty} \sum\limits_{j=1}^D r_j  \, .$$
  \item The other $N-1$ nodal points are ``internal'', i.e., for each $z_{\ell}\in Z(f)$ there exists $1\leq j \leq D$ such that $z_{\ell}\in ]0, L_j[ $ and $f$ is an anti-symmetric step function around it with length $d_{\ell}>0$, as in the interval case of Section \ref{sec:interval}.
\end{enumerate}
Therefore, one needs to optimize only the positive step-widths $r_1, \ldots , r_D$ and $d_1, \ldots, d_{N-1}$. We summarize these results in the following theorem:
\begin{theorem}
\label{thm:starequi}
Suppose $S_D$ has at least two edges longer than $c_1/2c_{\infty}$. Question \ref{q:minim}, i.e., minimizing $W_p^p(f_+, f_-)$ over $X(c_{\infty}, c_1, N , S_D)$, is equivalent to finding $r_1, \ldots , r_D, d_1, \ldots, d_{N-1} >0$ and a configuration function $q:\{1,\ldots ,N-1\} \to \{1, \ldots, D\}$ which minimize
\begin{subequations}\label{eq:mini_star_gen}
\begin{equation}\label{eq:g_star_gen}
    %\{r_j\}_{j=1}^D ~\, , \, ~\{d_{\ell}\}_{\ell=1}^{N-1} ~~\text{minimize} \qquad \qquad 
    \sum\limits_{\ell=1}^{N-1} c_{\infty} d_{\ell}^{p+1} + W_p^p(g_+, g_- ) \, ,
\end{equation}
where
\begin{equation}
    g(x) \equiv c_{\infty}\sum\limits_{j=1}^M \mathbbm{1}_{[0, r_j]}(x) - c_{\infty}\sum\limits_{j=M+1}^D \mathbbm{1}_{[- r_j,0]}(x) \, ,
\end{equation}
subject to the mass conservation conditions
\begin{equation}\label{eq:mass_consSd}
    \left\{\begin{array}{cc}
         c_{\infty}\sum\limits_{j=1}^M r_j = c_{\infty}\sum\limits_{j=M+1}^D r_j   \equiv \frac12\tilde{c}_1 \, , & \, \, \\
         \tilde{c}_1 + 2c_{\infty}\sum\limits_{\ell=1}^{N-1}d_{\ell} = c_1 \, , &
    \end{array}\right.
\end{equation}
and the admissibility constraints
\begin{equation}
\label{eq:lengthfea}
         r_j + 2\sum\limits_{q(\ell)=j} d_{\ell} \leq L_j \, , \qquad  1\leq j \leq D \, .
\end{equation}
\end{subequations}
\end{theorem}
Note that, since $g$ is given explicitly by the equivalence {\rev established in Lemma \ref{lem:star_vertex}}, see \eqref{eq:gstar_def}, and since it is a function on the line, $W_p(g_+, g_-)$ is given by the respective inverse CDFs as in \eqref{eq:wp_icdf}. Therefore optimizing $W_p^p(g_+, g_-)$ involves a direct, closed form calculation (see below). Moreover, if the edges are long, i.e., $L_j> c_1/c_{\infty}$ for any $1\leq j\leq D$, then the first constraint can be satisfied for any assignment of $r_j$'s and nodal point, i.e., for any choice of $q$. Hence, \eqref{eq:mini_star_gen} becomes computable in closed form {\em without} having to enumerate over all $(N-1)^D$ configuration functions $q$. %Conversely, we can say that for a given star $S_D$, the problem becomes computable for $c_1/c_{\infty}$ sufficiently small.

We can also see what is the complication introduced by short edges - if e.g., $L_D \to 0$, then effectively the number of effective edges becomes $D-1$ (in the sense that only vanishingly small mass can be assigned to $I_D$), thus changing the solution. Hence, obtaining a closed-form expression for \eqref{eq:mini_star_gen} can be algebraically daunting. In fact, already for $N=1$, where $Z(f)=\{0\}$, there are some complications. We will work out a few cases.
\subsection{Long edges, ${\bf N=1}$}\label{sec:star_longN1}

We find the minimizers of $W_p(f_+,f_-)$ on $f\in X(c_{\infty}, \tilde{c}_1, N=1, S_D)$ where all edges are sufficiently long.\footnote{The notation of the $L^1$ norm by $\tilde{c}_1$ here will be convenient when we consider $N\geq 1$ in the next section.} As discussed above, it is always preferable to have the nodal point at $0$. Moreover, it would always be better to distribute the $L^1$ mass on all edges. This is because the $L^1$ and $L^{\infty}$ constraints imply that ``wasted'' edges (on which $f_j =0$ everywhere) lead to transport over larger distances, and so to a more expensive optimal transport cost. Hence, $W_p(f_+,f_-)=W_p(g_+,g_-)$, the equivalent function on the line given by \eqref{eq:g_star_gen}. Since the edges are sufficiently long, if $f$ has $1\leq M<D$ positive edges, then $g$ could be a general function with $\|g\|_{\infty} = \max (D, D-M) c_{\infty}$. Hence, by our analysis for the interval (Section \ref{sec:interval}), the optimal $g$ is of the form:
$$
g(x) = \left\{\begin{array}{ll}
    -(D-M)c_{\infty} \, , & x\in (-r_-, 0) \, ,  \\
     Dc_{\infty} \, , & x\in (0, r_+) \,  , \\
     0 \, , & {\rm otherwise} \, .
\end{array} \right.
$$
Crucially, an $f:S_D\to \R$ to which this $g$ is equivalent is only possible when all of $S_D$ edges (the different $I_j$'s) are sufficiently long. In that case, the corresponding $f$ will be of the form
\begin{equation}\label{eq:fstar_steps}
    f_j(x) = \left\{ \begin{array}{lll}
    c_{\infty} \, , & x\in (0, r_+) \, , &1\leq j \leq M  \, , \\
     -c_{\infty} \, , & x\in (0, r_-) \,  , & M+1\leq j \leq D \, , \\
     0 \, , & {\rm otherwise} \, ,
\end{array} \right.
\end{equation} where $r_+$ and $r_-$ are determined by the mass conservation relation \eqref{eq:mass_consSd}, which reduces to
\begin{equation}\label{eq:mass_SdN1}
Mr_+ = (D-M)r_-     \, .
\end{equation}
%We will make this assumption first, and comment on the other cases later in this Section.
In these settings, minimizing $W_p (f_+, f_-)=W_p (g_+,g_-)$ reduces to finding the optimal parameters $r_+, r_->0$ and $1\leq M\leq D$.
\begin{proposition}[$N=1$ on a star]
Suppose $S_D$ has at least two edges longer than $c_1/2c_{\infty}$. Then, 
$${\rm arg}\min W_p ^p (f_+ , f_-) \qquad \quad {\rm over} \qquad \quad  X_s(c_{\infty},\tilde{c}_1, 1, S_D)  \, ,$$
is given by \eqref{eq:fstar_steps}
with $M=\lfloor D/2 \rfloor$, the lower integer part of $D/2$.
\end{proposition}

\begin{proof}
%Since Lemma \ref{lem:steps} really applies on transport plans, i.e., couplings, in sub-intervals, it is immediately applicable to the star. Furthermore, since we are considering the case of a single nodal point ($N=1$), we only need to treat the three following scenarios: the case of an ``internal'' point, i.e., a step function which is entirely supported on one of the intervals $I_j$, the case of $Z(f) = \{0\}$, and the case where the nodal point is on one of the intervals, but the step function ``crosses'' onto the other intervals through $0$. 

%The case of an internal point is already covered in Section \ref{sec:interval}, since the optimal transport takes place on a single interval. If $Z(f)=\{0\}$, then by Lemma \ref{lem:star_vertex} the transport of $f_+$ to $f_-$ through $0$ is equivalent to the transport of $g_+$ to $g_-$ through $0$ on an interval.

  %What would be the optimal value of $M$, i.e., the one for which $W_p(f_+, f_-)$ is minimal? Denote the overall $L^1$ mass of $f_+$ and $f_-$ as $\tilde{c}$. Following Lemma \ref{lem:star_vertex}, we can calculate $W_p(g_+,g_-)$ instead, 
  Since the equivalent $g$ (see Lemma \ref{lem:star_vertex}) is a function on an interval, we can use the explicit formula in terms of the inverse CDFs, \eqref{eq:wp_icdf}. By direct calculation, the inverse CDFs, defined on the interval $[0,\tilde{c}_1/2]$, are $G_+^{-1}(t)= r_+2\tilde{c}_1^{-1}(t)$ and $G_-^{-1}(t)= -r_- +r_-2\tilde{c}_1^{-1}t$. Assume without loss of generality that $r_+\geq r_-$, and so $\alpha \equiv 2\tilde{c}_1^{-1}(r_+-r_-) \geq 0 $. Hence,
\begin{align*}
    W_p^p(g_+, g_-) &= \int\limits_{0}^{\tilde{c}_1/2}   \left[ r_+2\tilde{c}_1^{-1}t - (-r_- +r_-2\tilde{c}_1^{-1}t) \right]^p  \, dt \\
  &=\int\limits_{0}^{\tilde{c}_1/2}  \left[ \alpha t +r_- \right]^p  \, dt  \numberthis \label{eq:wpp_alpha} \geq 
\int\limits_{0}^{\tilde{c}_1/2}  \left[ r_- \right]^p 
 \, .
\end{align*}
Hence, $W_p^p(g_+,g_-)$ is minimized when $\alpha = 0$, i.e., when $r_+=r_-$, which by the mass conservation relation \eqref{eq:mass_SdN1} yields $M=D/2$.

{\bf When $D$ is even}, substitution of $M=D/2$ into \eqref{eq:wpp_alpha} yields

\begin{equation}\label{eq:wp_star0_even}
    D~{\rm even} \qquad \qquad W_p^p(g_+,g_-)= \frac{\tilde{c}_1}{2} r_-^p =  \frac{1}{2    D^p}\left(\frac{\tilde{c}_1}{c_{\infty}}\right)^p \tilde{c}_1 \, , \qquad r_- = r_+ = \frac{\tilde{c_1}}{c_{\infty}}\frac{1}{D} \, ,
\end{equation}
which is the same as in the single step function case for $D=2$, see \eqref{eq:singleBump}. 

%\subsection{Long edges, $N=1$, and $D$ is odd}
{\bf When $D$ is odd}, $M=D/2$ is not an integer, and therefore not an admissible choice. We claim that the optimal choices are either $M=\lfloor D/2 \rfloor = (D-1)/2$ or $M=\lceil D/2 \rceil$. To see that, first note that the integrand in \eqref{eq:wpp_alpha} is monotonic in $\alpha$. Moreover, by the mass conservation relation \eqref{eq:mass_SdN1}, we get that
\begin{equation}\label{eq:alpha_M}
    \alpha = \alpha( M) =  \frac{1}{c_{\infty}}\left( \frac{1}{M} - \frac{1}{D-M} \right) \, ,
\end{equation} 
for any $   1\leq M\leq D$. Hence, $W_p^p (g_+, g_-)$ is monotonic convex in $M$ near its minimum at $D/2$, and the closest integer values, $\lfloor D/2 \rfloor$ and $\lceil D/2 \rceil$, minimize the cost.

To find the minimal cost, first, by direct computation of the integral in \eqref{eq:wpp_alpha}, we get
\begin{equation}\label{eq:wp_star0_odd_pre}
     W_p^p(g_+,g_-)=  \frac{1}{(p+1)\alpha} \left[\left(\frac{\alpha\tilde{c}_1}{2}+r_- \right)^{p+1}-r_-^{p+1} \right]  \, .
\end{equation}
By substitution of $M=\lfloor D/2 \rfloor= (D-1)/2$ into \eqref{eq:alpha_M}, we get that 
\begin{align*}
    \alpha \left( \lfloor \frac{D}{2}\rfloor\right) =  \frac{1}{c_{\infty}}\left( \frac{2}{D-1} - \frac{2}{D+1} \right)
    %\\ 
   % &=  \frac{2}{c_{\infty}} \frac{D+1-(D-1)}{(D-1)(D+1)} \\ &
    = \frac{4}{c_{\infty}(D-1)(D+1)} \, .
\end{align*}
Similarly, by the same constraint in \eqref{eq:wpp_alpha}, we have that
%\footnote{{Maybe we can also present the value of $r_+$ below?}}
$$r_- c_{\infty} \frac{D+1}{2} = \frac{\tilde{c}_1}{2} \, \, \Longrightarrow r_- = \frac{\tilde{c}_1}{c_{\infty}(D+1)},\quad\text{and} \quad
r_+=
\frac{\tilde{c}_1}{c_{\infty}(D-1)} \, .$$
Substituting the expressions for $\alpha$ and $r_-$ into \eqref{eq:wp_star0_odd_pre}, we get
\begin{align*}
    W_p^p(g_+,g_-)&=  %\frac{1}{(p+1)}
    \frac{c_{\infty}(D-1)(D+1)}{4(p+1)}\left[\left(\frac{4\tilde{c}_1}{2c_{\infty}(D-1)(D+1)}
    %\frac{\tilde{c}_1}{2}
    +\frac{\tilde{c}_1}{c_{\infty}(D+1)} \right)^{p+1}-\left(\frac{\tilde{c}_1}{c_{\infty}(D+1)}\right)^{p+1} \right] \\
%    &=\frac{(D-1)}{4(p+1)(D+1)^p} \left(\frac{\tilde{c}_1}{c_{\infty}}\right)^p \tilde{c_1}\left[\left(\frac{2}{D-1}+1 \right)^{p+1}-1 \right] \\
%    &=\frac{(D-1)}{4(p+1)(D+1)^p(D-1)^{p+1}} \left(\frac{\tilde{c}_1}{c_{\infty}}\right)^p \tilde{c_1}\left[\left(2+(D-1) \right)^{p+1}-(D-1)^{p+1} \right] \\
    &=\frac{1}{4(p+1)(D+1)^p(D-1)^p} \left(\frac{\tilde{c}_1}{c_{\infty}}\right)^p \tilde{c_1}\left[(D+1)^{p+1} - (D-1)^{p+1} \right] \, . %\numberthis %\label{eq:wp_star0_odd}
\end{align*}
To compare this expression to the even case \eqref{eq:wp_star0_even}, let us set $p=1$. Then
\begin{equation}\label{eq:wp_star0_odd}
W_1 (g_+,g_-)= \frac{D}{2(D+1)(D-1)} \frac{\tilde{c}_1}{c_{\infty}} \tilde{c_1} 
=
\frac{\tilde{c}_1^2}{4\tilde{D}
c_{\infty}}
\, , 
\end{equation}
with $\tilde{D}=\frac{(D+1)(D-1)}{2D}$ as defined before.
Since
$$ \frac{1}{2(D+1)} 
\leq \frac{1}{4\tilde{D}}  
\leq \frac{1}{2(D-1)}\, ,$$
we easily see that $W_1$ lies in between the optimal cost for stars with $D-1$ and $D+1$. This is consistent with the intuition that as $D$ increases, the cost decreases.

Either $D$ is even or odd, the overall transport cost is cheaper
than the case of an internal point. This makes intuitive sense, because a nodal point on the vertex allows for more mass to be transported over a shorter distance, due to Lemma \ref{lem:star_vertex}. %Now, suppose $z\in I_j$ but $f$ is supported on multiple intervals, i.e., $z$ is very close to $0$. Then by the same token, the mass needs to be distributed over a {\em longer} distance than in the case of $Z_f=\{0 \}$, since some of the transport has to be done over a single interval. To conclude, the minimizers of the transport are of the form \eqref{eq:fstar_def}.
\end{proof}

\subsection{Long edges, ${\bf N\geq 1}$, ${\bf D}$ is even}\label{sec:star_longDeven}
In principle, now one can go to the general case of $N\geq 1$ and, using the minimization formulation \eqref{eq:mini_star_gen}, find the minimizers. Comparing the even and odd cases in \eqref{eq:wp_star0_even} and \eqref{eq:wp_star0_odd}, respectively, it seems that the case of $D$ even is more tractable.

The total cost of transport through the $0$-vertex depends on a single parameter $r_-=r_+$, which is not yet determined.   Equivalently, it depends on $\tilde{c}_1$, the amount of mass concentrated around the vertex, with $0\leq \tilde{c}_1 \leq c_1$. Since we assume that the edges are sufficiently long, then (based on our analysis for the interval in Section \ref{sec:interval}) the minimization problem \eqref{eq:mini_star_gen} simplifies by $d_1= \cdots = d_{N-1} = d_{\rm int}$. We therefore wish to minimize
\begin{equation}\label{eq:star_prew1}
W_p^p(f_+,f_-) = (N-1)c_{\infty}d_{\rm int}^{p+1}  + \frac{D}{2}c_{\infty}r_+^{p+1} \, ,
\end{equation}
subject to 
$$2(N-1)d_{\rm int}c_{\infty} + Dr_+ c_{\infty} = c_1 \, .$$
As in the case of the interval (Section \ref{sec:interval}), we use the method of Lagrange multipliers
$$\mathcal{L} = (N-1)c_{\infty}d_{\rm int}^{p+1}  + \frac{D}{2}c_{\infty}r_+^{p+1} - \lambda (2(N-1)d_{\rm int}c_{\infty} + Dr_+ c_{\infty} -c_1) \, .$$
Then 
$$ 0 = \partial_{d_{\rm int}} \mathcal{L} = (p+1)(N-1)c_{\infty}d_{\rm int}^p - 2\lambda (N-1)c_{\infty} \quad \Longrightarrow \quad d_{\rm int}^p = \frac{2\lambda}{p+1} \, ,   $$
and
$$ 0 = \partial_{r_+} \mathcal{L} = (p+1)
\frac{D}{2}c_{\infty}r_+^p - \lambda Dc_{\infty} \quad \Longrightarrow \quad r_+^p = \frac{2\lambda}{p+1} \, .  $$
In particular, $r_+=d_{\rm int}$. The constraint associated with $c_1$ then yields
$$d_{\rm int}=r_+ = \frac12\frac{1}{(N-1)+\frac{D}{2}}\frac{c_1}{c_{\infty}} $$
and the minimal optimal transport cost is 
$$W_p (f_+, f_-) = \frac{c_1}{c_{\infty}} \frac{c_1^{\frac{1}{p}}}{2^{1+\frac{1}{p}}} \frac{1}{\tilde{N}} \, ,  \qquad \tilde{N}\equiv N-1+
%\tilde{D},\quad \tilde{D}\equiv
\frac{D}{2} \, .$$
In comparison with Theorem \ref{thm:w1_min} for the interval, we see that for a star with an {\em even number of edges} that are sufficiently long, $N$ is replaced by $\tilde{N}$, i.e., the vertex as a nodal point is equivalent to $%\tilde{D}=
D/2$ nodal points on a line.

\subsection{Long edges, ${\bf N\geq 1}$, ${\bf D}$ is odd}\label{sec:star_longOdd} In this case, the same argument yields a much more cumbersome expression, and so we resolve it only for $p=1$. Lagrange multipliers method now yields
$$
\mathcal{L} =  (N-1)c_{\infty}d_{\rm int}^{2}  + 
\frac{1}{4\tilde{D}} \frac{\tilde{c}_1^2}{c_{\infty}
}  - \lambda \left(2(N-1)d_{\rm int}c_{\infty} + \tilde{c}_1 - c_1 \right) \, 
,\quad\text{for }\tilde{D}\equiv\frac{(D+1)(D-1)}{2D}.
$$
Here, since $r_+ \neq r_-$, it is easier to minimize the cost with respect to $\tilde{c}_1$, the total $L^1$ mass around the vertex. Then
 
$$0 = \partial_{\tilde{c}_1} \mathcal{L} = \frac{1}{2\tilde{D}} \frac{\tilde{c}_1}{c_{\infty}} -\lambda  \, ,  $$
and so
$$ 0 = \partial_{d_{\rm int}} \mathcal{L} = 2(N-1)c_{\infty}d_{\rm int} - 2\lambda (N-1)c_{\infty} \quad \Longrightarrow \quad d_{\rm int} = \lambda = \frac{1}{2\tilde{D}} \frac{\tilde{c}_1}{c_{\infty}} \, .   $$
The $L^1$ constraint now reads as
$$\frac{N-1}{\tilde{D}}
%2(N-1) \frac{D}{(D+1)(D-1)}
\frac{\tilde{c}_1}{c_{\infty}} c_{\infty} +\tilde{c}_1 = c_1 \quad \Longrightarrow \quad 
\tilde{c}_1 = \frac{\tilde{D}}{ N-1 +\tilde{D}} c_1 \,. $$
As a sanity check, we see that $\tilde{c}_1 \to c_1 $ as $D\to \infty$.
We note that, in this case, the sharp lower bound for $f\in X(c_{\infty}, c_1, N, S_D)$ reads like
\begin{align*}
W_1(f_+,f_-) &\geq (N-1)c_{\infty}d_{\rm int}^2 + \frac{1}{4\tilde{D}} \frac{\tilde{c}_1^2}{c_{\infty}}  \\
&\geq 
\frac{c_1^2}{4 (N-1+\tilde{D})c_{\infty}} 
\\
&\geq \frac{1}{4\tilde{N}} \frac{c_1^2}{c_{\infty}} \, , \qquad \text{for}\quad \tilde{N}\equiv N-1+\tilde{D} \, . \numberthis \label{eq:sharp_lb_DoddN}
\end{align*}
Compared with the case where the number of edges is even (see Sec.\ \ref{sec:star_longDeven}), $\tilde{D}$ may be viewed as the effective multiplicity of the nodal point at  the vertex.

The key difference between the lower bounds on $W_1$ for stars and the lower bounds obtained for non-branching spaces (the interval, the circle, and general $d$-dimensional surfaces in \eqref{eq:uncert}, see \cite{carroll2020enhanced, cavalletti2021indeterminacy, sagiv2020transport}), is that here the lower bound {\em does not admit a multiplicative `uncertainty principle'-like structure} of $4 W_1 \cdot N \geq \frac{c_1}{c_\infty}c_1$, where $A$ depends on the domain and norms of $f$. Nevertheless, for the cases in Sections \ref{sec:star_longDeven} and \ref{sec:star_longOdd}, we may still say that the structure is preserved with an {\em effective} $N$ that is replaced by $\tilde{N}$ to account for the particular geometry of the nodal point at the vertex of the star.
However, discussions below imply that the precise characterization of $\tilde{N}$ could be much more involved than only the degree $D$ with short edges present.

\subsection{${\bf D=3}$ with one short edge, ${\bf N\geq 1}$ and
${\bf p=1}$}\label{sec:starD3_short}
In the long-edge cases discussed above, the minimization problem
given in Theorem \ref{thm:starequi} is assumed to have a solution in the interior of the constraint set \eqref{eq:lengthfea}. To see further how the geometry affects the conclusion on the minimal optimal transport cost, we consider the special case of $D=3$, with two long edges and a third short one as an illustrative example. Our goal is to see how the limit $L_3\to 0^+$ interpolates between $S_3$ with three long edges, and $S_2\cong [-L_1, L_2]$.

First, note that to optimize the transport cost in the present case, it remains advantageous to include the node of the star in $Z(f)$. Moreover, the maximal amount of mass should be assigned to the short edge if its length is sufficiently small. Meanwhile, 
 away from the node, we retain the symmetric structure around the other points in $Z(f)$ that are located on the long edges. Near those points, we utilize earlier discussion and use $d_{int}$ to represent the length of the subintervals on which the functions are taken to be $\pm c_\infty$ with the total optimal transport cost given by
 $(N-1) c_{\infty}  d_{int}^2$.
 
 As before, to get an explicit expression for \eqref{eq:g_star_gen}, we need to work out $W_1 (g_+,g_-)$ as a function of the overall mass assigned to the vertex, $\tilde{c}_1$. Assume without loss of generality that $f\geq 0$ on $I_1$ and $f\leq 0$ on $I_2$ and $I_3$ near the node. For $L_3 \ll 1$, then the maximal amount of mass should be assigned to $I_3$, i.e., $r_3=L_3$ and $f(x)=-c_{\infty}$ on $I_3$. By the consideration above, we have that the equivalent $g:\R \to \R$ is of the form 
 $$ g(x) = \left\{\begin{array}{ll}
      c_{\infty} & x\in (-r_1, 0)  \, ,\\
      -2c_{\infty} &  x\in (0, L_3) \, , \\
      -c_{\infty} & x\in (L_3, r_2) \, ,\\
      0 & {\rm otherwise} \, .
 \end{array} \right. $$
 for two undetermined constants $r_1, r_2 > 0$. For $p=1$, we can use the CDF counterpart of \eqref{eq:wp_icdf}, $W_1 (g_+, g_-)= \|G_+-G_-\|_{L^1(\R)}$, where $G_{\pm}$ are the CDFs of $g_{\pm}$. By direct computation
 $$G_+(t) - G_-(t) = \left\{\begin{array}{ll}
      c_{\infty}(t+r_1) & t\in (-r_1, 0)  \, ,\\
      c_{\infty}r_1 -2c_{\infty}t &  t\in (0, L_3) \, , \\
      c_{\infty}r_1 -2c_{\infty}L_3 -c_{\infty}t & x\in (L_3, r_2) \, .
 \end{array} \right. $$
Since the overall mass conservation \eqref{eq:mass_consSd} reduces to $r_1=r_2+L_3=\tilde{c}_1/2c_{\infty}$, we have that 
\begin{align*}
    W_1 (g_+,g_-) &= \|G_+ - G_-\|_{L^1(\R)} \\
    &= \frac14 \tilde{c}_1 r_1 +\frac14 \tilde{c}_1 r_2 +\frac14 \tilde{c}_1L_3 - r_2c_{\infty}L_3 \\
    &= \frac{\tilde{c}_1^2}{4c_{\infty}} - r_2c_{\infty}L_3 \\ 
    &= \frac{\tilde{c}_1^2}{4c_{\infty}} - \left(\frac{\tilde{c}_1}{2c_{\infty}}-L_3\right)c_{\infty}L_3 \, .
    \end{align*}
Putting these considerations together, the minimization problem in \eqref{eq:mini_star_gen} reduces to
\begin{align*}
W_1(f_+, f_-) & = (N-1) c_{\infty}  d_{\rm int}^2 
+\frac{\tilde{c}_1^2}{4c_{\infty}} - \left(\frac{\tilde{c}_1}{2c_{\infty}}-L_3\right)c_{\infty}L_3 \\
&= c_{\infty} \left[
(N-1)  d_{\rm int}^2 
+\left(\frac{\tilde{c}_1}{2c_{\infty}} -
\frac{L_3}{2} 
\right)^2+
\frac{3}{4}L_3^2
\right]
\, ,
  \end{align*}
subject to the $L^1$ mass conservation constraint 
$$ 2(N-1)d_{\rm int}  + 
 \frac{\tilde{c}_1}{c_{\infty}} = \frac{c_1}{c_{\infty}} \, .$$
Thus, for $\beta \equiv L_3c_{\infty}/c_1$, we get the optimal transport cost 
\begin{equation}
\label{eq:d3cost}
 W_1(f_+, f_-) \geq c_{\infty}\left[\frac{1}{4N}\left( \frac{c_1}{c_{\infty}}-L_3 \right)^2 + \frac34L_3^2 \right]
 = \frac{c_1^2}{4c_{\infty}}
\frac{(1-\beta)^2 + 3N\beta^2}{N} 
\end{equation}
with equality for
$$d_{\rm int}
=\frac{1}{2N}
\left( \frac{c_1}{c_{\infty}}  -L_3
\right), \;\text{ and }\;
\frac{\tilde{c}_1}{c_{\infty}}
=
\frac{c_1}{Nc_{\infty}} + 
\frac{(N-1)
L_3}{N}\,. 
$$

\subsubsection*{$\beta$ as an interpolation parameter between $S_3$ and an interval.}
We can understand these expressions between the two limits of
either shortening or lengthening $I_3$, 
i.e., as the star graph deforms into an interval with $L_3\to 0^+$
on one hand and as 
%\begin{enumerate}
 %   \item {\bf Shortening $I_3$,} i.e., as the star graph deforms into an interval. As $L_3 \to 0^+$, the right hand side of \eqref{eq:d3cost}, i.e., $\min w_1(f_+,f_-)$, reduces to $c_1^2 / 4Nc_{\infty}$. This is exactly the same lower bound we have for the interval in \eqref{eq:uncertnew} with $p=1$. 
 %  \item {\bf Lengthening $I_3$,} and in particular, the regime where the 
third edge becomes long enough for us to recover the case of three long edges, as described in Sections \ref{sec:star_longN1} and \ref{sec:star_longOdd}. Note that our analysis involving the short edge holds precisely until the case where there is no distinction between $I_2$ and $I_3$ in terms of mass allocation. By the results in Section \ref{sec:star_longOdd}, the optimal mass around the origin for $D=3$ % (for which $\tilde{D}=4/3$
    is $\tilde{c}_1 =4c_1/(3N+1)$. Since near the vertex, $L_3$ and $L_2$ both contain the support of $f_-$, the threshold length for $L_3$ to enable the optimal $\tilde{c}_1$ thus corresponds to $\beta^* = c_{\infty}L_3/c_1 = \tilde{c}_1 /4c_1 = 1/(3N+1)$.
    %Hence
Indeed, we may rewrite the right hand side of
    \eqref{eq:d3cost} to get an equivalent form given by
    \begin{align*}
     W_1(f_+, f_-) \geq  
        \frac{c_1^2}{4c_{\infty}} \frac{3N+1}{N}\left(\beta-\frac{1}{3N+1}\right)^2
        +\frac{c_1^2}{4c_{\infty}(N+\frac{1}{3})} \, .
    \end{align*}
The lower bound is clearly a decreasing function of $\beta$ for $\beta \in [0, 1/(3N+1)]$, hence {\em interpolating}, as expected,  between the lower bound derived in \eqref{eq:uncertnew} for  $p=1$ and $\beta=0$ (i.e., $D=2$ for the case of an interval) and that given in \eqref{eq:sharp_lb_DoddN} with $D=3$ and $\tilde{D}=\frac{4}{3}$.

Meanwhile, for  small but nonzero
$L_3$,
the non-multiplicative nature of the lower bound on the right hand side of
\eqref{eq:d3cost} is evident. Compared to the cases in Sections \ref{sec:star_longDeven} and \ref{sec:star_longOdd}, we see that the multiplicative form is lost not only in terms of $N$ but also with respect to $\frac{c_1}{c_{\infty}}$.
%\footnote{{A question: can we define $\frac{N}{(1-\beta)^2 + 3N\beta^2}+1-N $ as an effective degree like $D/2$ and $\tilde{D}$ for the long edge cases, thus showing the more involved nature with $\beta$ depending on geometry and the norms of the function? need to show it is positive.}}
\begin{remark*}
What can be said on stars such as in Fig.\ \ref{fig:pathologicalStar}, for which Theorem \ref{thm:starequi} does not apply? We outline the strategy for their analysis: first, the equivalence relation in Lemma \ref{lem:star_vertex} can be extended to a single nodal point {\em near} the vertex, on the long edge. The adjacency and monotonicity arguments (Lemmas \ref{lem:adj_sd} and \ref{lem:mono_sd}) seem to go through as well.  Then, the optimization problem in Theorem~\ref{thm:starequi} should be extended by adding another parameter to represent the location of that ``special'' nodal point, which is near the vertex. Finally, we remark that the condition for a long edge being $\geq c_1/2c_{\infty}$ is sufficient, but not necessarily sharp. The sharpness of this lower bound would become one interesting element in the study of Question \ref{q:minim} on general metric graphs.
\end{remark*}

%Another 

\section{Outlook}
\label{sec:outlook}
An important message of this work is that the original problem of minimizing the transport cost over the specified function class is equivalent to a generalized minimal surface problem associated with {\em optimal domain partition}.
For one-dimensional domains, minimizing the cost reduces to optimally positioning the nodal points and locating masses around them. We believe that this approach can be generalized to more challenging cases, some of which are discussed below.

\subsubsection*{Multidimensional domains}
 
In an analogous way to the one-dimensional settings of this study, it seems that on bounded and regular domain $\Omega \subseteq \R^d$, the minimizers will be step functions concentrated near the interfaces between $f_+$ and $f_-$. If this is indeed the case,  the key remains to be the solution to the minimization problem
\begin{equation}\label{eq:multid}
\min\limits_{f \in X_s(1,c_1/c_{\infty},N, \Omega )} W_p(f_+, f_-) .
\end{equation}
%Similar to the one-dimensional settings, we can define the subspace $X_s(1,c_1,N, \Omega )\subset X$ of  piecewise constant functions taking values only $\pm 1$ and $0$, i.e., a subspace of signed indicator functions. 
It should be commented that, for step functions, the existence of a lower bound  (\eqref{eq:uncert}  \cite{carroll2020enhanced, cavalletti2021indeterminacy, sagiv2020transport}) combined with nucleation arguments such as \cite[Lemma 2.8]{novack2021least}, should yield the existence of minimizers. An element in $X_s(1,c_1/c_\infty,N, \Omega )$ corresponds to a unique partition of $\Omega $ into three disjoint subsets {\rev $\Omega_+, \Omega_-, \Omega_0 \subset \Omega $, where $\Omega_0 = \Omega \setminus (\Omega_+ \cup \Omega _-)$}, such that $$|\Omega_-|=|\Omega_+|=\frac{c_1}{2} \, , \qquad \mathcal{H}^{d-1} (\Gamma)=N \, , \qquad \Gamma \equiv  \partial \Omega_+ \cap \partial \Omega _-  \, ,$$ assuming that the partitions are sufficiently regular such that $\mathcal{H}^{d-1}(\Gamma)$ is well defined. Denoting the set of such partitions by $Y(c_1,N, \Omega)$, {\rev one may attempt a similar approach to the one presented in this work:  reducing the functional minimization problem \eqref{eq:multid}  to a purely geometric one:}
\begin{equation}
\label{eq:minsurf}
\min\limits_{f \in X(1,c_1,N, \Omega)} W_p(f_+, f_-) = 
\min\limits_{(\Omega_-, \Omega_+, \Omega_0)\in Y(c_1,N,\Omega)} W_p(\mathbbm{1}_{\Omega_+}, %-
\mathbbm{1}_{\Omega_-}) 
\end{equation}
Clearly, the latter becomes an optimal partition problem or a minimal surface problem. That is, the problem under consideration may be formulated as an equivalent geometric question.
\begin{question}\label{q:minimgeo}
For a bounded and regular domain $\Omega\subset \R^d$, $p\geq 1$, $c_1>0$, and 
 $N>0$, what are the optimal partitions $(\Omega_-, \Omega_+, \Omega_0)$ of $\Omega$ 
 and the corresponding minimum value associated with the problem
\begin{equation}\label{eq:minpart}
\min\limits_{(\Omega_-, \Omega_+, \Omega_0)\in Y(c_1,N, \Omega)} W_p(\mathbbm{1}_{\Omega_+}, %-
\mathbbm{1}_{\Omega_-})
    \end{equation}
\end{question}

This draws interesting connections to possibly other optimal domain partition problems such as those related to different optimal transport minimization problems, e.g., partitions defined by the optimal quantization (centroidal Voronoi tessellations) \cite{du1999sirev,merigot2021non}. One may find more discussions on other optimal domain partition problems in \cite{bucur1998existence}.  A variation of Question \ref{q:minimgeo} 
considered in \cite{buttazzo2020wasserstein, novack2021least,peletier2009partial,xia2021existence} is to modify the set of feasible partitions by removing the constraint on the prescribed measure of the nodal, and consider the minimization of $|\partial \Omega_+|+\lambda W_1(\mathbbm{1}_{\Omega_+}, %-
\mathbbm{1}_{\Omega_-})$ where $|\partial \Omega_+|$ is the Hausdorff measure of the boundary set of $\Omega_+$ and $\lambda>0$ is a penalty constant.

The solution to  the purely geometric Question \ref{q:minimgeo} can be used to offer a new perspectives on the existing studies concerning \eqref{eq:uncert} in multiple-dimensions  \cite{steinerberger2020metric,sagiv2020transport,carroll2020enhanced,cavalletti2021indeterminacy} and 
to make further connections with %notions like the perimeter of nodal sets \cite{buttazzo2020wasserstein, novack2021least, xia2021existence}, as well as
the bilayer membrane limits \cite{peletier2009partial,
lussardi2014variational}.
In particular, the optimal transport problem admits a decomposition into a continuous family of one-dimensional optimal-transport problems. One of the main challenges is that, as we seek to 
solve \eqref{eq:minpart}, this decomposition changes as well.
Based on the discussion on the one dimensional examples, it is obvious that the optimal partitions may not be unique in general, even though they yield the same optimal transport cost.
While a complete solution is beyond the scope of this work, one might conjecture that $\Gamma$ is made of piecewise planar surfaces, which are the conventional minimal surfaces in the Euclidean space.

\subsubsection*{Metric graphs and Laplacian eigenfunctions} This work is a starting point for the study of the Wasserstein minimization problems on metric graphs, for which the star $S_D$ (Section \ref{sec:trees}) is perhaps the most basic nontrivial example. Our analysis for $S_D$ suggests what might be the key challenges for general metric graphs - first, that the associated minimization problems might be algebraically complicated, and second, that the equivalence principles (Lemma~\ref{lem:star_vertex}) {\rev are} not easily extended to general graphs. The second key challenge is to find a monotonicity-like property which generalizes Lemma \ref{lem:mono_sd} to non-star graphs.

While the study of optimal transport on metric graphs is relatively new \cite{erbar2021gradient}, the study of nodal sets of special functions on metric graphs is a well established field, see e.g., \cite{alon2018nodal, band2014nodal, band2012number, berkolaiko2014stability, colin2013magnetic, gnutzmann2003nodal, hofmann2021pleijel}, as well as \cite{berkolaiko2013introduction} and the references therein. In its heart, the question is a natural extension of Sturm's Oscillations Theory: given the $N$-th Laplacian eigenfunction on a certain metric graph, how many zeroes does it have? 

Independently, uncertainty principles for the Wasserstein distance such as \eqref{eq:uncertnew} were applied to generalize Sturm and Sturm-Hurwitz (sums of Laplacian eigenfunctions) type results in dimensions $d\geq 1$ \cite{carroll2020enhanced, cavalletti2021indeterminacy, sagiv2020transport, steinerberger2020metric}. Generally speaking, the approach proceeds in two steps: first, prove a lower bound of the type \eqref{eq:uncertnew}, and then, obtain an upper bound on $W(f_+,f_-)$ for the specialized type of functions under consideration, e.g., Laplacian eigenfunction. The current work might therefore open the way of connecting these two lines of works, though many open challenges remain before such a connection is established.

\subsubsection*{Minimization over more specialized function classes} Concerning the possible connection to properties of Laplacian eigenfunctions,
related studies on lower bounds and uncertainty-principles of Laplacian eigenfunctions on Riemannian manifolds (and ${\rm RCD}$ spaces in general) can be found in \cite{deponti2022indeterminacy, mukherjee2021sharp, steinerberger2021wasserstein}.
One may note that the key to get the new lower bounds like \eqref{eq:uncertnew} presented in this work is through the study of the minimization problem over the function class defined by, e.g.,~\eqref{eq:setX}. The latter class is large enough to allow minimizers taking on the form of step functions. On the other hand, eigenfunctions of elliptic operators are smooth. Thus, another natural research direction to explore is the study of  possibly different bounds, associated with similar minimization problems, but over classes of functions that are more regular or of special forms. In particular,
the special forms may correspond to discrete function spaces, and the problem of passing from 
discretizations of $L^{\infty}$ functions, where the minimization problem is finite-dimensional, to the continuum limit (studied in this paper), might be of independent interest.

\subsubsection*{Functional analytic perspective}
Inequalities such as \eqref{eq:uncert} and \eqref{eq:uncertnew} can be thought of as 
interpolation inequalities. On the lower-bound side, there are the 
$L^1$ and $L^{\infty}$ norms involving no derivatives. 
The upper bound consists of two terms, which can be
connected to derivative norms of different orders, respectively: first, $|Z(f)|$ is smaller than $\|D\mathbbm{1}_{{\rm supp}(f_+)}\|_1$, where the derivative $D$ is taken in the sense of functions of bounded variations \cite{evans2018measure, giusti1984minimal}; then, $W_p(f_+, f_-)$ can be viewed as norms of derivatives of negative order, e.g., the $W_2$ distance is 
related to the standard Sobolev $\dot{H}^{-1}$ norm; See \cite{loeper2006uniqueness, peyre2018comparison} for details. 
 
While this functional-analytic perspective is not used in this paper, it can lead to the exploration  of more general variational problems of the type Question \ref{q:minim}, e.g., by allowing for constraints not only in $L^1$ and $L^{\infty}$, but in other norms such as $L^p$ norms for $1<p<\infty$. Moreover, understanding \eqref{eq:uncert} as an interpolation inequality, \eqref{eq:uncertnew} can be thought of as a sharp interpolation inequality with optimal constants, of which an extensive literature exists \cite{del2002best, del2003optimal}.

\section*{Acknowledgments}
The authors would like to thank Ram Band for many invigorating discussions and help provided along the way, as well as for suggesting to study this minimization problem on metric graph and highlighting some of the challenges in that area. The authors would also like to thank Raghavendra Venkatraman for bringing to our attention an issue with one of the definitions in a previous version of this manuscript. Many thanks to Fabio Cavalletti and Tim Laux for useful comments and references. Thanks also to Zirui Xu for bringing references \cite{lussardi2014variational,peletier2009partial} to our attention. The research of Q.D.\ is supported in part by NSF DMS-1937254, DMS-2012562 and CCF-1704833. A.S.\ acknowledges the support of the AMS-Simons Travel Grant, and is supported in part by Simons Foundation
Math + X Investigator Award \#376319 (Michael I. Weinstein).

\subsection*{Data availability} Data sharing not applicable to this article as no datasets were generated or analysed during
the current study
\subsection*{Conflict of interest}  The authors have no conflicts of interest to declare that are relevant to the content of this article.

\bibliographystyle{amsplain}
\bibliography{bigBib}

\end{document}